\documentclass{article}

\usepackage[english]{babel}
\usepackage[utf8]{inputenc}
\usepackage[T1]{fontenc}

\pdfoutput=1

\usepackage[bookmarks,bookmarksdepth=2, colorlinks=true, linkcolor=blue,citecolor=red, urlcolor=blue]{hyperref}

\usepackage{fullpage}

\usepackage{mystyle}
\usepackage{wrapfig}
\usepackage{tikz-cd}
\usepackage{multicol}
\usepackage{graphicx}
\usepackage[numbers]{natbib}
\usepackage{enumitem}
\usepackage{textgreek}
\usepackage{multirow}

\numberwithin{equation}{section}

\begin{document}
\centerline{\bf COX RINGS OF ALMOST HOMOGENEOUS $\mathbf{\SL_2}$-THREEFOLDS}
\bigskip

\centerline{\small ANTOINE VEZIER}
\bigskip

\begin{abstract}
We study Cox rings of normal threefolds on which $\SL_2$ acts with a dense orbit. Exploiting the method of $U$-invariants, we obtain combinatorial criteria for the total coordinate space and the base variety to have log terminal singularities. Then, we investigate the iteration of Cox rings and obtain a bound on its length. Finally, we develop a general approach to the description of the Cox ring by generators and relations which is effective for normal $\SL_2/\mu_n$-embeddings.
\end{abstract}

\section{Introduction}
The \textit{complexity} of the action of a connected reductive group is a very important birational invariant in studying the geometry of the action. This is the minimal codimension of an orbit of a Borel subgroup. The normal algebraic varieties of complexity zero are the \textit{spherical varieties}, they constitute a natural generalization of toric varieties and have a well understood geometry (\cite[Chap. 5]{TimashevBook}). The next step is the study of varieties of complexity one, which is not so well developed, except for the case of actions by tori (\cite{KKMS}, \cite{HausenIteration}, \cite{LiendoSuss}). In general, two cases are possible for a normal variety of complexity one: either it is \textit{almost homogeneous} (i.e. there exists a dense orbit), or it admits a one-parameter family of spherical orbits. The first examples of the former case consist of almost homogeneous $\SL_2$-threefolds, that is, normal $\SL_2$-varieties of dimension three with a dense orbit. Various aspects make this class of examples especially important, and one can view their study as a preliminary step toward a better understanding of general almost homogeneous varieties of complexity one. For example, these examples yield all homogeneous spaces of complexity and rank one via "parabolic induction", as shown by Panyushev in \cite{Panyushev2}.

An almost homogeneous $\SL_2$-threefold $X$ can be viewed as a \textit{normal embedding} of a homogeneous space of the form $\SL_2/F$ where $F$ is a finite subgroup. This means that the orbit morphism associated to a point in the dense orbit factors through a $\SL_2$-equivariant open immersion $\SL_2/F\xhookrightarrow{} X$. This point of view is interesting because one can take advantage of the combinatorial data defining the embedding to approach various questions. The setup of this combinatorial framework goes back to the seminal work of Luna and Vust (\cite{LunaVust}). It has been considerably clarified in the complexity zero case by Brion, Knop, Luna, Vust, and by Timashev in the complexity one case (\cite{TimashevClassification}). In fact Timashev's description doesn't restrict to almost homogeneous spaces but allows to classify varieties in any fixed equivariant birational class in terms of objects of convex geometry. The combinatorial description of normal $\SL_2/F$-embeddings has been obtained by Luna and Vust (\cite{LunaVust}) when $F$ is trivial, Moser-Jauslin (\cite{LMJ}) who extended the classification of Luna-Vust for arbitrary $F$, and Timashev (\cite{TimashevClassification}) who put these results in his framework. 

In the present work, we mainly focus on \textit{Cox rings} of normal $\SL_2/F$-embeddings. The Cox ring of a normal variety $X$, denoted $\cox(X)$, is an important invariant that encodes a lot of geometric information, see \cite{coxrings} for a comprehensive reference on this rich subject. It is the ring of global sections of the \textit{Cox sheaf}, which is, roughly speaking, the direct sum 
\begin{center}
$\Rr_X:=\bigoplus_{[\Ff]\in\clg(X)}\Ff$
\end{center}
indexed by elements of the \textit{class group} of $X$, i.e. the group of isomorphism classes of divisorial sheaves on $X$. Under mild assumptions, $\Rr_X$ can be endowed with a structure of quasi-coherent $\clg(X)$-graded $\Oo_X$-algebra, whence a structure of $\clg(X)$-graded ring on $\cox(X)$. When $\Rr_X$ is of finite type as an $\Oo_X$-algebra, its relative spectrum $\hat{X}$ is a normal $\Gamma_{\clg(X)}$-variety over $X$, where $\Gamma_{\clg(X)}$ is the diagonalizable group with character group $\clg(X)$. This is called the \textit{characteristic space} of $X$, and the structural morphism $\hat{X}\rightarrow X$ is a good quotient by $\Gamma_{\clg(X)}$. For a finitely generated Cox ring, its spectrum $\tilde{X}$ is called the \textit{total coordinate space} of $X$, and the affinization morphism $\hat{X}\rightarrow \tilde{X}$ is a $\Gamma_{\clg(X)}$-equivariant open immersion whose image has a complement of codimension $\geq 2$ in $\tilde{X}$. When $X$ is smooth, the characteristic space is a $\Gamma_{\pic(X)}$-torsor over $X$ called the \textit{universal torsor}.

Let $X$ be a normal $\SL_2/F$-embedding. In fact, we consider the \textit{equivariant Cox ring} of $X$ introduced and studied in \cite{avezier_EqCoxRings}. For the case of $\SL_2$, it is canonically isomorphic to the ordinary Cox ring (\cite[2.3.4]{avezier_EqCoxRings}). However, we take advantage of the structure of graded $\SL_2$-algebra provided by the construction of the equivariant Cox ring. Also, we occasionally use general results from \cite{avezier_EqCoxRings}. In particular, $\cox(X)$ is a finitely generated normal domain, and we use the description of the $k$-subalgebra of $U$-invariants $\cox(X)^U$, where $U$ is the unipotent part of the standard Borel subgroup $B$ of $\SL_2$. For the convenience of the reader, we recall in Section \ref{Sec_GeneralitiesComplexityOne} useful facts on normal rational varieties of complexity one and their (equivariant) Cox ring. In Section \ref{SecSL2Threefolds}, we describe the class group of $X$ by generators and relations. Then, we provide a combinatorial criterion for $X$ to have log terminal singularities, extending a previous result of Degtyarev (\cite{Degtyarev}).

The study of the Cox ring of $X$ starts in Section \ref{Sec_CharacterizingLogTerminality}. As an application of \cite[3.4.3]{avezier_EqCoxRings}, we obtain a criterion of combinatorial nature for the total coordinate space $\tilde{X}$ to have log terminal singularities. This is an interesting question, particularly since the work of Gongyo, Okawa, Sannai and Takagi characterizing varieties of Fano type via singularities of Cox rings (\cite{Gongyo}).

In Section \ref{Sec_GeoSpecialFiber}, we put to light some interesting new phenomena with regard to the \textit{special fiber}, i.e. the schematic fiber at zero of the quotient morphism
\begin{center}
$\tilde{X}\xrightarrow{//\SL_2} \AAA_k^N$,
\end{center}
where $N$ is the number of $\SL_2$-invariant prime divisors in $X$ (\cite[2.8.5]{avezier_EqCoxRings}). The general fibers of this morphism are isomorphic to the total coordinate space of the dense orbit (\cite[2.8.1]{avezier_EqCoxRings}). For a spherical variety, the quotient morphism is faithfully flat, and the special fiber is a normal affine variety which is moreover \textit{horospherical}, in the sense that the product of two irreducible representations, say $V_\lambda$, $V_\mu$, in its coordinate algebra is their Cartan product $V_{\lambda+\mu}$ (\cite[3.2.3]{Brion2007}). These nice geometric properties are important ingredients for the determination of a presentation of the Cox ring of a spherical variety in loc. cit. We show by examples that these properties do not extend to the complexity one world. Nevertheless, we give a criterion for the special fiber to be a normal variety, in terms of the basic geometry of $X$.

Pursuing our investigation, we turn to the study of the \textit{iteration of Cox rings} for $X$ in Sections \ref{SecUTorsor} and \ref{Sec_Iteration}. Roughly, iteration of Cox rings consists in studying the sequence of total coordinate spaces
\begin{center}
$...\rightarrow\tilde{X}^{(n)}\rightarrow\tilde{X}^{(n-1)}\rightarrow...\rightarrow\tilde{X}^{(2)}\rightarrow\tilde{X}$,
\end{center}
where $\tilde{X}^{(n)}$ denotes the total coordinate space of $\tilde{X}^{(n-1)}$. A basic question is whether this sequence is finite, in which case $X$ is said to have \textit{finite iteration of Cox rings}, and the last obtained Cox ring is called \textit{the master Cox ring}. By virtue of \cite[3.5.1]{avezier_EqCoxRings}, $X$ has finite iteration of Cox rings with a factorial master Cox ring. More precisely, we obtain
\begin{prop*}
The length of the iteration of Cox rings sequence of $X$ is bounded by $4$.
\end{prop*}
\noindent See the statement \ref{Prop_BoundedIteration} where more precise bounds are given depending on the finite subgroup $F\subset\SL_2$. Also, we obtain a commutative diagram
\begin{center}
\begin{tikzcd}
\tilde{X}^{(m)} \arrow[d] \arrow[r] & ... \arrow[r] & \tilde{X}^{(2)} \arrow[r] \arrow[d] & \tilde{X} \arrow[d] \\
\tilde{Y}^{(m)} \arrow[r]           & ... \arrow[r] & \tilde{Y}^{(2)} \arrow[r]           & \tilde{Y},          
\end{tikzcd}
\end{center}
where the horizontal arrows are structural morphisms of characteristic spaces, the vertical arrows are \textit{almost principal $U$-bundles}, and all squares are cartesian. Moreover, $\tilde{X}^{(m)},\tilde{Y}^{(m)}$ are factorial and $\tilde{Y}$ is the total coordinate space of a certain variety of complexity one under the standard maximal torus of $\SL_2$. 

Finally, we develop in Section \ref{SecPresentationCox} an approach for the description of the Cox ring by generators and relations. In the same way as it is important for projective geometry to have explicit homogeneous coordinates, it is an important problem to find an explicit description by generators and relations of the Cox ring. Moreover, the shape of the equations defining the total coordinate space turns out to have important arithmetic consequences for the base variety, see for example the recent preprint \cite{ManinPeyre} related to the Manin-Peyre conjecture. Our approach is effective for normal $\SL_2/\mu_n$-embeddings. This eventually leads us to compare our treatment with previous work by Batyrev and Haddad in the particular case of affine almost homogeneous $\SL_2$-threefolds (\cite{Haddad}).

\begin{ack*}
I am very grateful to Michel Brion for his help, and a very careful reading of the present work. I also thank Jürgen Hausen for interesting discussions during my stay in Tübingen, and Lukas Braun for enlightening email exchanges on a preliminary version of this paper. 
\end{ack*}

\begin{conv*}
Let $k$ denote an algebraically closed base field of characteristic zero. In this text, we work in the category of \textit{algebraic schemes}, that is, separated $k$-schemes of finite type. Hence, a morphism $(f,f^\sharp):X \rightarrow Y$ of algebraic schemes is meant to be a $k$-morphism. A \textit{variety} is an integral algebraic scheme. Without further precision, a \textit{point} of a variety is meant to be a closed point. The sheaf of \textit{regular functions} (or structure sheaf) on an algebraic scheme $X$ is denoted $\Oo_X$. A \textit{subvariety} of a variety $X$ is a locally closed subset equipped with its reduced scheme structure. The sheaf of units associated to $\Oo_X$ is denoted $\Oo_X^*$. 

In this text, an \textit{algebraic group} is an affine algebraic group scheme. The character group of an algebraic group $G$ is denoted $\hat{G}$ or $X^*(G)$. Given a finitely generated abelian group $M$, we let $\Gamma_M$ denote the \textit{diagonalizable (algebraic) group} $\spec(k[M])$ with character group $M$. Without further precision, the letter $\TT$ denotes an arbitrary torus. An \textit{almost homogeneous variety} is a normal variety on which a connected algebraic group acts with a dense orbit. 

Let $G$ be an algebraic group, $X$ an algebraic $G$-scheme, and $q:X\rightarrow  Y$ a $G$-invariant morphism. We say that $q$ is a $G$-\textit{torsor} (or \textit{principal $G$-bundle}) over $Y$ if $q$ is faithfully flat, and the natural morphism $G\times X\rightarrow X\times_Y X$ is an isomorphism. As algebraic groups are smooth in characteristic zero, $q$ is faithfully flat if and only if it is smooth and surjective. We say that $q$ is a \textit{trivial $G$-torsor} over $Y$ if there is a $G$-equivariant isomorphism $X\simeq G\times Y$ over $Y$, where $G$ acts on $G\times Y$ via left multiplication on the first factor. A fact that will be used implicitly many times in the text is that a torsor under a torus or a unipotent algebraic group is locally trivial in the Zariski topology.

Let $G$ be an algebraic group and $X$ an affine algebraic $G$-scheme. If $\Oo(X)^G$ is a finitely generated $k$-algebra, then the affine algebraic scheme $X//G:=\spec(\Oo(X)^G)$ is called the categorical quotient of $X$ by $G$. It is the universal object in the category of $G$-invariant morphisms from $X$ to affine algebraic schemes.

Our convention is that a \textit{reductive group} is a \textit{linearly reductive group}, that is, every finite dimensional representation of such group is semisimple. In particular, a reductive group is not necessarily connected. Let $G$ be a reductive group and $X$ an algebraic $G$-scheme. A \textit{good quotient} of $X$ by $G$ is an affine $G$-invariant morphism $q:X\rightarrow Y$ such that $q^\sharp$ induces an isomorphism $\Oo_Y\rightarrow (q_*\Oo_X)^G$. It is a universal object in the category of $G$-invariant morphisms from $X$ to algebraic schemes.

Let $X$ be a normal variety, the \textit{group of Weil divisors} is denoted $\wdiv(X)$. Every Weil divisor defines a coherent sheaf $\Oo_X(D)$ whose non-zero sections over an open subset $U$ are rational functions $f\in k(X)^*$ such that $\divi(f)+D$ defines an effective divisor on $U$. A \textit{divisorial sheaf} on $X$ is a coherent reflexive sheaf of rank one. The \textit{class group} (resp. \textit{Picard group}) of $X$, denoted $\clg(X)$ (resp. $\pic(X)$), is the group of isomorphism classes of divisorial sheaves (resp. invertible sheaves) on $X$. It is isomorphic to the group of Weil divisors (resp. Cartier divisors) modulo linear equivalence through the morphism $[D]\mapsto [\Oo_X(D)]$. Let $G$ be an algebraic group acting on a normal variety $X$. The \textit{equivariant class group} (resp. \textit{equivariant Picard group}) of a normal $G$-variety $X$ is denoted $\clg^G(X)$ (resp. $\pic^G(X)$), and the equivariant Cox ring is denoted $\cox^G(X)$ (see \cite[Sec. 2.2 and 2.3]{avezier_EqCoxRings}). A \textit{pointed normal variety} $(X,x)$ consists of a normal variety $X$ and a smooth point $x\in X$.  For each class $ [\Ff]\in \clg^G(X)$, there exists a canonical representative $\Ff^x$ called the \textit{rigidified $G$-linearized divisorial sheaf} associated to $[\Ff]$ (see \cite[Sec. 2.3]{avezier_EqCoxRings}).

The \textit{canonical sheaf} $\omega_X$ on a normal variety $X$ is the pushforward on $X$ of the sheaf of differental forms of maximal degree on the smooth locus $X_{sm}$. It is a divisorial sheaf, and any Weil divisor $K_X$ such that $\Oo_X(K_X)\simeq\omega_X$ is a \textit{canonical divisor}. A \textit{$\QQ$-Gorenstein variety} is a normal variety such that some non-zero power of the canonical sheaf is invertible. A \textit{Gorenstein variety} is a normal Cohen-Macaulay variety whose canonical sheaf is invertible. 

Let $X$ be a normal variety. We say that $X$ has \textit{rational singularities} if there exists a proper birational morphism
\begin{center}
$\varphi:Z\rightarrow X$,
\end{center}
where $Z$ is a smooth variety (a resolution of singularities), and such that $R^i\varphi_*\Oo_Z=0$, $\forall i>0$. This last property doesn't depend on the choice of a resolution. We say that $X$ has \textit{log terminal singularities} if the following conditions are satisfied:
\begin{itemize}[label=$\bullet$]
\item $X$ is $\QQ$-Gorenstein.
\item There is a resolution of singularities $\varphi:Z\rightarrow X$ such that
\begin{center}
$K_{Z}=\varphi^*K_X+\sum\alpha_i E_i$, $\alpha_i>-1$,
\end{center}
where the sum runs over the exceptional divisors $E_i$ of $\varphi$.
\end{itemize}
\end{conv*}

\section{Normal rational varieties of complexity one}
\label{Sec_GeneralitiesComplexityOne}

Let $G$ be a connected reductive group, and $(X,x)$ be a pointed normal $G$-variety of complexity one. Suppose moreover  that $X$ is a rational variety, so that its Cox ring is finitely generated (\cite[3.1.4]{avezier_EqCoxRings}; this is verified e.g. when $X$ is almost homogeneous). Fix a Borel subgroup $B$, a maximal torus $T$ in $B$, and denote by $U$ the unipotent part of $B$. Suppose that $G$ has trivial Picard group, so that any divisorial sheaf on $X$ admits a $G$-linearization (see \cite[2.2.2]{avezier_EqCoxRings}; this can always be achieved by replacing $G$ with a finite cover). In this section, we recall facts on the geometry of $X$ following \cite{TimashevBook}, \cite{Ponomareva}, and \cite{avezier_EqCoxRings}. 

\subsection{$B$-stable divisors}
\label{Sec_BStabledivs}

By a theorem of Rosenlicht (\cite[5.1]{TimashevBook}) there is a rational quotient
\begin{center}
$\pi:X\dashrightarrow \PP^1_k$
\end{center}
by $B$. Thus, general $B$-orbits determine a one-parameter family of $B$-stable prime divisors in $X$. This rational map is defined by two global sections $a,b$ of a rigidified $G$-linearized divisorial sheaf $\Ff^x$ on $X$ (\cite[Sec. 3.2]{avezier_EqCoxRings}). The pullback of Weil divisors on $\PP^1_k$ corresponds to the usual pullback of Cartier divisors and is given by
\begin{center}
$\pi^*:\wdiv(\PP_k^1)\rightarrow \wdiv(X), p=[\alpha:\beta]\mapsto\divi(\beta a-\alpha b)$,
\end{center}
where $[\alpha:\beta]$ are homogeneous coordinates of $p\in\PP^1_k$.
All the $B$-stable prime divisors in $X$ but a finite number lie in the image of $\pi^*$ (\cite[Sec. 3.1]{avezier_EqCoxRings}).

\begin{defn}\cite[3]{Ponomareva}
The prime divisors in $X$ lying in the image of $\pi^*$ are the \textit{parametric divisors}. The finite set of $B$-stable prime divisors that are not parametric is the set of \textit{exceptional divisors}. For an exceptional divisor $E$, the image of $\pi_{|E}$ is either dense or a point in $\PP^1_k$. In the former case, we say that $E$ \textit{dominates} $\PP^1_k$, in the latter case, the image point is called \textit{exceptional}.
\end{defn}

\subsection{The algebras $\cox^G(X)^U$ and $\cox(X)^U$}
\label{SecCoxU}

Suppose that $\Oo(X)^*\simeq k^*$, so that both $\cox^G(X)$ and $\cox(X)$ are well defined and finitely generated (\cite[3.1.4]{avezier_EqCoxRings}). We recall the description of the $k$-algebra $\cox^G(X)^U$ of $U$-invariants obtained by Ponomareva in \cite[Thm. 4]{Ponomareva} and generalized in \cite[3.2.3]{avezier_EqCoxRings}. Also, considering the canonical structure of $U$-algebra on $\cox(X)$, we recall an interpretation of $\cox(X)^U$ as the Cox ring of a complexity one $T$-variety (\cite[3.3.2]{avezier_EqCoxRings}).

\begin{nota}
Let $(x_i)_{i\in I}$ be the finite family of exceptional points with respective homogeneous coordinates $[\alpha_i:\beta_i]$. For all $i\in I$, let $(E^{x_i}_j)_j$ be the finite family of exceptional divisors that are sent to $x_i$ by $\pi$. Let $(E_k)_k$ be the finite family of exceptional divisors dominating $\PP^1_k$. Equip $\Oo(E^{x_i}_j)^x$ (resp. $\Oo(E_k)^x$) with arbitrary $G$-linearizations, and let $s_{ij}$ (resp. $s_k$) denote the canonical sections of these sheaves associated with the divisors $E^{x_i}_j$ (resp $E_k$), and let $h_{ij}$ denote the (integral) coefficient of $E^{x_i}_j$ in the divisor $\pi^*(x_i)$.
\end{nota}

\begin{thm}\label{ThmUinvCoxG}\cite[3.2.3]{avezier_EqCoxRings}
The $k$-algebra $\cox^G(X)^U$ is generated as a $k[\hat{G}]$-algebra by the elements $a,b,(s_{ij})_{ij}, (s_k)_k$. The ideal of relations contains the following identities
\begin{center}
$\beta_i a-\alpha_i b=\lambda_i\prod_j s_{ij}^{h_{ij}},$
\end{center}
where for all $i\in I$, $\lambda_i$ is a certain character of $G$. If moreover, the condition
\begin{itemize}[label=$(\star)$]
\item the common degree of the sections $a$ and $b$ is $\ZZ$-torsion free in $\clg^G(X)\times\hat{T}$
\end{itemize}
is satisfied, then the above relations generate the whole ideal.
\end{thm}

\begin{rem}\cite[3.2.4]{avezier_EqCoxRings}\label{Rem_ThmUinvCoxG}
Examples of situations where the condition $(\star)$ is satisfied are given by
\begin{itemize}
\item rational normal $\TT$-varieties of complexity one such that $\Oo(X)^{\TT}\simeq k$.
\item almost homogeneous varieties of complexity one.
\end{itemize}
\end{rem}

In \cite[3.4.2]{coxrings}, are studied the coordinate algebras of certain \textit{trinomial varieties}, that is, affine varieties which are the intersection in an affine space of hypersurfaces defined by trinomial equations. These algebras are obtained via a construction taking in input certain matrices $A$ and $P_0$ storing the coefficients and exponents of the trinomials. Their spectrum defines normal affine varieties of complexity one under the action of (the connected component) of a diagonalizable group, and such that the invariant regular functions are constant. Moreover, these algebras turn out to be Cox rings of varieties of complexity one under the action of a torus \cite[3.4.3]{coxrings}. 

\begin{cons}\cite[3.4.2.1]{coxrings}\label{ConsRingFactoC1}
Fix integers $r\in\ZZ_{\geq 1}$, $m\in\ZZ_{\geq 0}$, a sequence of integers $n_0,...,n_r\in\ZZ_{\geq 1}$, and let $n:=n_0+...+n_r$. Consider as inputs

\begin{itemize}[label=$\bullet$]
\item A matrix $A:=[a_0,...,a_r]$ with pairwise linearly independent column vectors $a_0,...,a_r\in k^2$.
\item An $r\times (n+m)$ block matrix  $P_0:=[L\,\,\, 0_{r,m}]$, where $L$ is an  $r\times n$ matrix  built from the $n_i$-tuples $l_i:=(l_{i1},...,l_{in_i})\in\ZZ^{n_i}_{\geq 1}$, $0\leq i\leq r$, called \textit{exponent vectors}, as below
\[
L = \begin{bmatrix} 
    -l_0& l_1 & \dots & 0 \\
    \vdots & \vdots & \ddots & \vdots \\
    -l_0 &  0  &  \dots   & l_r 
    \end{bmatrix}
\]
\end{itemize}
Now consider the polynomial algebra $k[T_{ij},S_k]$, where $0\leq i\leq r$, $1\leq j\leq n_i$, and $1\leq k\leq m$. For every $0\leq i\leq r$, define a monomial
\begin{center}
$T_i^{l_i}:=T_{i1}^{l_{i1}}...T_{in_i}^{l_{in_i}}$,
\end{center}
whence the name "exponent vector".
Denote $\mathscr{I}$ the set of triples $(i_1,i_2,i_3)$ with $0\leq i_1< i_2< i_3\leq r$, and for all $I\in\mathscr{I}$, consider the trinomial
\[g_I:=\det \begin{bmatrix} 
    T_{i_1}^{l_{i_1}}& T_{i_2}^{l_{i_2}} & T_{i_3}^{l_{i_3}} \\
    a_{i_1} &  a_{i_2}  &  a_{i_3} 
    \end{bmatrix}.
\]
We introduce a grading on $k[T_{ij},S_k]$ by the abelian group $K_0:=\ZZ^{n+m}/\Im(^tP_0)$, where $^tP_0$ is the transpose of  $P_0$. Let $Q_0:\ZZ^{n+m}\rightarrow K_0$ be the projection, and set
\begin{center}
$\deg T_{ij}:=Q_0(e_{ij})$, and $\deg S_k:=Q_0(e_k)$,
\end{center}
where $(e_{ij}, e_k)$ is the standard basis of $\ZZ^{n+m}$. Finally consider the $K_0$-graded $k$-algebra
\begin{center}
$R(A,P_0):=k[T_{ij},S_k]/(g_I)_{I\in\mathscr{I}}$.
\end{center}
\end{cons}

\begin{prop}\cite[3.3.2]{avezier_EqCoxRings}\label{Prop_RAP0}
Suppose that the condition $(\star)$ of \ref{ThmUinvCoxG} is satisfied. Then, the $k$-algebra $\cox(X)^U$ is isomorphic to an algebra $R(A,P_0)$ constructed as in \ref{ConsRingFactoC1}. 
\end{prop}

\begin{rem}\label{Rem_Inputdata}
By \cite[3.3.2]{avezier_EqCoxRings}, the input data to be used in Construction \ref{ConsRingFactoC1} in order to obtain $\cox(X)^U$ can interpreted geometrically:
\begin{itemize}[label=$\bullet$]
\item $m$ is the number of exceptional divisors in $X$ dominating $\PP^1_k$,
\item $a_0,...,a_r$ are homogemeous coordinates on $\PP^1_k$ of the exceptional points $x_0,...,x_r$.
\item the exponent vectors are the vectors formed by the multiplicities of the exceptional divisors in the pullbacks $\pi^*(x_i)$, $i=0,...,r$.
\end{itemize}
\end{rem}

\begin{rem}\label{Rem_Platonic}
Consider the spectrum $\tilde{Y}$ of an algebra $R(A,P_0)$ viewed as the total coordinate space of a normal rational variety $Y$ of complexity one under a torus action. The geometries of $\tilde{Y}$ and $Y$ highly depend on the exponent vectors $l_i$, $i=0,...,r,$ involved in the equations. In \cite{HausenIteration}, the study of singularities on $Y$ leads to the following notion that will be used later: a ring $R(A,P_0)$ is a \textit{Platonic ring} either if $r\leq 1$, or if every tuple $(l_{0i_0},...,l_{ri_r})$ is \textit{Platonic}, i.e. after ordering it decreasingly, the first triple is one of the \textit{Platonic triples}
\begin{center}
$(5,3,2),(4,3,2),(3,3,2),(x,2,2),(x,y,1)$, $x\geq y\geq 1$,
\end{center}
and the remaining integers of the tuple equal one.
\end{rem}

\subsection{Combinatorial classification}
\label{SecCombinatorialMaterial}

In this section, we present the combinatorial framework for the classification of complexity one normal rational $G$-varieties lying in a fixed $G$-birational class. Our reference for this material is \cite{TimashevBook}. Let $K$ be the field of rational functions of $X$. Recall that a \textit{geometric valuation} of $K$ is a discrete valuation $K^*\rightarrow\QQ$ of the form $\alpha v_D$, where $\alpha\in\QQ_+$, and $v_D$ is the normalized discrete valuation associated to a prime divisor $D$ in a normal variety $Z$ whose field of fractions is identified with $K$.

\begin{nota}
We set $\Dd=\Dd(K)$ for the set of prime divisors of $X$ that are not $G$-stable, and $\Dd^B$ the subset of $B$-stable ones. This last subset consists of the so-called \textit{colors} of $K$. These sets don't depend on $X$ up to $G$-birational equivalence and we identify them with the corresponding sets of normalized geometric valuations of $K$. Let $\Vv$ denote the set of $G$-valuations of $K$, that is, the geometric $G$-invariant valuations of $K$. Let $\Vv(X)$ denote the subset of normalized $G$-valuations corresponding to $G$-stable prime divisors of $X$. The element of $\Vv(X)$ corresponding to a $G$-stable prime divisor $D$ is denoted $v_D$.
\end{nota}

There is an exact sequence of abelian groups
\begin{center}
$1\rightarrow (K^B)^*\simeq k(\PP^1_k)^*\rightarrow K^{(B)}\rightarrow \Lambda(X) \rightarrow 1$,
\end{center}
where $K^{(B)}$ is the multiplicative group of rational $B$-semi-invariant functions and $\Lambda(X)$ is the associated group of weights. The abelian group $\Lambda(X)$ is the \textit{weight lattice} of the $G$-variety $X$. It is a subgroup of $\hat{T}$ whose rank is the \textit{rank} of the $G$-variety $X$. After choosing a splitting $\lambda\mapsto f_\lambda$ of the above sequence, we view $\Lambda(X)$ as a submodule of $K^{(B)}$. Then, considering a geometric valuation $v$ of $K$, the restriction $v_{|K^{(B)}}$ is determined by a triple $(x,h,l)$, where $x\in\PP^1_k$, $h\in\QQ_+$, and $l\in \Ee:=\Hom(\Lambda(X),\QQ)$. Indeed, the restriction of $v$ to $k(\PP^1_k)$ is again a geometric valuation \cite[B.8]{TimashevBook}, hence of the form $hv_x$, where $x\in\PP^1_k$. On the other hand, $v_{|\Lambda(X)}$ yields an element $l\in\Ee$.

\begin{defn}
For all $x\in \PP^1_k$, consider the closed half-space $\Ee_{x}=\QQ_+\times \Ee$. The \textit{hyperspace} $\breve{\Ee}$ associated to $K$ is the union of the $\Ee_{x}$ glued together along $\Ee$.
\end{defn}

The set $\Vv$ embeds in $\breve{\Ee}$, and $\Vv_x:=\Vv\cap\Ee_x$ is a simplicial convex polyhedral cone for all $x\in\PP^1_k$. Also, the natural map $\varrho:\Dd^B\rightarrow\breve{\Ee}$ is not injective in general. Notice that from the preceding section, all but a finite number of $B$-stable prime divisors are sent in $\breve{\Ee}$ to a vector of the form $\varepsilon_x:=(x,1,0)\in \Ee_x$ for a certain $x\in \PP^1_k$. 

\begin{defn}
The pair $(\Vv,\Dd^B)$ is the \textit{colored equipment} of $K$. We say that $(\breve{\Ee},\Vv,\Dd^B,\varrho)$ is the \textit{colored hyperspace} of $K$.
\end{defn}

\begin{defn}\label{DefHypercone}
A \textit{cone} in $\breve{\Ee}$ is a cone in some $\Ee_x$, $x\in\PP^1_k$. A \textit{hypercone} in $\breve{\Ee}$ is a union $\Cc=\cup_{x\in\PP_k^1}\Cc_x$ of convex cones, each generated by a finite number of vectors, and such that
\begin{enumerate}
\item $\Cc_x=\Kk+\QQ_+\varepsilon_x$ for all $x\in\PP^1_k$ but a finite number, where $\Kk:=\Cc\cap\Ee$.
\item One of the following cases occurs
 \begin{description}
  \item (A) $\exists x\in\PP^1_k$, $\Cc_x=\Kk$.
  \item (B) $\emptyset\neq \Bb:=\sum\Bb_x\subset\Kk$, where $\varepsilon_x+\Bb_x=\Cc_x\cap(\varepsilon_x+\Ee)$.
\end{description}
\end{enumerate}
The hypercone is said of \textit{type} (A) (resp. (B)) depending on the alternative of condition $2$.
The hypercone is said \textit{strictly convex} if  each $\Cc_x$ is, and $0\notin \Bb$.
\end{defn}

\begin{defn}\label{DefHyperconeEngendre}
Let $\Qq\subset \breve{\Ee}$ a subset all of whose elements but a finite number are of the form $\varepsilon_x$ for some $x\in\PP^1_k$. Let $\varepsilon_x+\Pp_x$ the convex hull of intersection points of the half-lines $\QQ_+q, q\in\Qq$ with $\varepsilon_x+\Ee$. We say that the hypercone $\Cc=\Cc(\Qq)$, where the $\Cc_x$ are generated by $\Qq\cap\Ee_x$ and $\Pp:=\sum\Pp_x$, is \textit{generated} by $\Qq$.
\end{defn}
The fundamental result is that strictly convex hypercones $(\Cc,\Rr):=\Cc(\Ww\cup\varrho(\Rr))$, where $\Ww\subset\Vv$, $\Rr\subset\Dd^B$, and $0\notin\varrho(\Rr)$, classify affine $B$-stable open subvarieties of normal $G$-models of $K$ (the so-called \textit{$B$-charts}). These hypercones are called \textit{colored hypercones}.

\begin{rem}\label{RemBChartTypeB}
Let $X_0$ be a $B$-chart defined by a colored hypercone $(\Cc,\Rr)$.  Then $\Oo(X_0)^B\simeq k$ if and only if $(\Cc,\Rr)$ is of type (B).
\end{rem}

A finite set of colored hypercones $(\Cc,\Rr)$ defines $B$-charts $\mathring{X}_i$ and $G$-models $X_i:=G\mathring{X}_i$. These $G$-models can be glued together into a $G$-model if and only if these colored hypercones defines a \textit{colored hyperfan}. In turn, the colored hyperfans classify normal $G$-models of $K$. To make this last notion precise we need the notion of a \textit{hyperface} of a hypercone, which is defined through the notion of a \textit{linear functional} on the hyperspace. We can think of the abelian group $K^{(B)}$ as the dual object to the hyperspace. Indeed, every $f=f_0f_\lambda\in K^{(B)}$ defines a so-called linear functional on $\breve{\Ee}$, namely the restriction to each $\Ee_x$ is a $\QQ$-linear form ($(h,\gamma)\mapsto h v_x(f_0)+\gamma(\lambda)$). Conversely, considering a linear functional on $\breve{\Ee}$, a multiple of it is given by a rational $B$-semi-invariant function.
\begin{defn}
A \textit{face} of a hypercone $\Cc$ is a face $\Cc'$ of a certain cone $\Cc_x$ such that $\Cc'\cap\Bb=\emptyset$. A \textit{hyperface} of $\Cc$ is a hypercone $\Cc'=\Cc\cap\ker\varphi$, where $\varphi$ is a linear functional on $\Cc$ such that $\varphi(\Cc)\geq 0$. If $\Cc$ is of type B, its \textit{interior} is by definition $\interior \Cc:=\cup_{x\in\PP^1_k}\interior\Cc_x\cup\interior\Kk$.
\end{defn}

\begin{defn}
A colored hypercone $(\Cc,\Rr)$ of type (B) is \textit{supported} if $\interior \Cc\cap \Vv\neq \emptyset$. A \textit{(hyper)face} of $(\Cc,\Rr)$ is a colored (hyper)cone $(\Cc',\Rr')$, where $\Cc'$ is a (hyper)face of $\Cc$, and $\Rr'=\Rr\cap\varrho^{-1}(\Cc')$. A \textit{colored hyperfan} is a set of supported colored cones and hypercones of type (B) whose interiors are disjoint inside $\Vv$, and which is obtained as the set of all supported colored (hyper)faces of finitely many colored hypercones.
\end{defn}

An interesting feature of this classification framework is that the lattice of $G$-stable subvarieties of $X$ can be read from the combinatorial representation of $X$ as a colored hyperfan. Indeed, consider a $G$-stable subvariety $Y\subset X$, and a $B$-chart $\mathring{X}$ intersecting $Y$ which is defined by a hypercone $\Cc(\Ww\cup \varrho(\Rr))$. Denote $\Vv_Y\subset \Ww$, $\Dd_Y^B\subset \Rr$ the respective subsets which correspond to the $B$-stable prime divisors in $\mathring{X}$ whose closure in $X$ contains $Y$, and let $\Cc_Y$ be the (hyper)cone spanned by $\Vv_Y\cup\varrho(\Dd_Y^B)$. Then for any other $G$-stable subvariety $Z$, we have $Y\subset \overline{Z}$ if and only if $(\Cc_Z,\Dd^B_Z)$ is a (hyper)face of $(\Cc_Y,\Dd^B_Y)$.

\subsection{Structure of $B$-charts}
\label{Sec_LocStructure_BChart}

The local structure theorem of Brion, Luna and Vust (\cite[Thm 1.4]{BLV}) is a very useful tool for the study of varieties with group action, the following variant is due to Knop.

\begin{thm}\cite[1.2]{Knop93}\label{ThmLocalStruct1}
Let $G$ a connected reductive group, $X$ a $G$-variety equipped with an ample $G$-linearized invertible sheaf $\Ll$, and $s\in\Ll(X)^{(B)}$ a $B$-semi-invariant section. Let $P$ the parabolic subgroup stabilizing $s$ in $\PP(\Ll(X))$, with Levi decomposition $P_u\rtimes L$, $T\subset L$. Then the open subvariety $X_s$ is $P$-stable, and there exists a closed $L$-stable subvariety $Y\subset X_s$ such that the morphism induced by the $P$-action
\begin{center}
$P_u\times Y\rightarrow X_s$
\end{center}
is a $P$-equivariant isomorphism.
\end{thm}

Consider a $B$-chart $V\subset X$, the $G$-variety $X_0:=GV$, and the $B$-stable effective divisor $D:=X_0\setminus V$. By \cite[Lemma 2]{TimashevCartier}, $\Oo_{X_0}(D)$ is an ample ($G$-linearized) line bundle on $X_0$. The canonical section associated with $D$ is $B$-semi-invariant. Thus, its projective stabilizer $P$ is a parabolic subgroup containing $B$. Then, the above theorem yields a $P$-equivariant isomorphism
\begin{center}
$P_u\times Y\rightarrow V$.
\end{center}

\begin{cor}\label{Cor_LocalStructThm}
Suppose that $G=\SL_2$ and $V$ is not $G$-stable. Then, there is a $B$-equivariant isomorphism
\begin{center}
$U\times Y\rightarrow V$,
\end{center}
where $Y$ is a normal affine $T$-variety of complexity one.
\end{cor}

\section{Almost homogenous $\SL_2$-threefolds}
\label{SecSL2Threefolds}

\subsection{Generalities}

\label{Sec_SL2Threefolds_Gen}

Consider the algebraic group $G:=\SL_2$ identified with its group of rational points, namely the matrices
\[
g=\begin{bmatrix}
    g_1 & g_2 \\
    g_3 & g_4
    \end{bmatrix}
\]
such that $g_i\in k$ and $\det(g)=1$. By abuse of notation, we also let $g_i$ denote the corresponding matrix coordinates on $G$. Consider also the Borel subgroup $B$ whose elements are upper triangular matrices, and the maximal torus $T\subset B$ whose elements are diagonal matrices. The character group of $T$ is free of rank one generated by
\[
\omega:=(\begin{bmatrix}
    t & 0 \\
    0 & t^{-1}
    \end{bmatrix}\mapsto t).
\]
In the sequel, we fix a normal embedding $i:G/F\hookrightarrow X$, where $F$ is a finite subgroup of $G$. This means that there exists a point $x$ in the $G$-variety $X$ such that the associated orbit morphism factors through a $G$-equivariant open immersion $i:G/F\xhookrightarrow{} X$. In the sequel, we identify $G/F$ with this dense orbit. It is well known that $F$ is conjugate either to a cyclic subgroup $\mu_n\subset T$ of order $n\geq 1$, or to one of the famous binary polyhedral groups (\cite[4.4]{SpringerInvariantTheory}):
\begin{itemize}[label=$\bullet$]
\item $F_{\DD_{n}}:=$ Binary dihedral group of order $4n$, $n>1$,
\item $F_\TT:=$ Binary tetrahedral group,
\item $F_\OO:=$ Binary octahedral group,
\item $F_\II:=$ Binary icosahedral group.
\end{itemize}
The pointed normal $G$-variety $(X,x:=i(F/F))$ is rational of complexity one, and the rational quotient $\pi:X\dashrightarrow \PP^1_k$ by $B$ induces a geometric quotient $\pi_{\mid G/F}:G/F\rightarrow\PP^1_k$. The situation is summarized in the commutative diagram
\begin{center}
	\begin{tikzcd}
		G \arrow[r,"B\backslash
"] \arrow[d,  "/F"] & B\backslash G\simeq \PP^1_k \arrow[d, "/F"]\\
		G/F \arrow[r, "B\backslash"] \arrow[d, "i", hook]& \PP^1_k/F\simeq \PP^1_k\\
		X \arrow[ur, "\pi", dashed] &
	\end{tikzcd}
\end{center}
where $B$ acts on $G/F$ by left multiplication. Denote $\Ff^x$ the unique $G$-linearized rigidified divisorial sheaf associated to $[\pi^*\Oo_{\PP^1_k}(1)]$, and $a,b\in\Ff^x(X)^{(B)}$ the two global sections corresponding to the pullbacks of (a choice) of coordinates on $\PP^1_k$. Also, let $n_0\omega$ be the common $B$-weight of $a,b$.

The colored equipment associated to $k(G/F)$ has been described by Timashev in \cite[Sec. 5]{TimashevClassification} for the various finite subgroups. For convenience, we recall the associated basic facts in Appendix \ref{SecAppendix}. When $F=\mu_n, n\geq 3$, the morphism $\pi_{|G/F}$ defines exactly two exceptional colors $E^{x_0}$, $E^{x_\infty}$ corresponding to the two $\mu_n$-fixed points $x_0, x_\infty$ of $\PP^1_k$. When $F$ is binary polyhedral, $\pi_{|G/F}$ defines exactly three exceptional colors $E^{x_v}$, $E^{x_e}$, and $E^{x_f}$. The subscripts are aimed to suggest (when $k=\CC$) vertices, edges and faces of the corresponding Platonic solids inscribed in $\PP^1_k$ identified with the Riemann sphere. These exceptional colors correspond exactly to the degenerate $F$-orbits in $\PP^1_k$ (see \cite[4.4]{SpringerInvariantTheory}). The homogeneous space $G/F$ being affine, the complement of $G/F$ in $X$ is a finite union of $G$-stable prime divisors and every $G$-stable prime divisor lies in this complement. By Section \ref{SecCombinatorialMaterial}, there is at most one $G$-stable prime divisor $X^{\infty}$ in $X$ dominating $\PP^1_k$. It consists of infinitely many $G$-orbits, all isomorphic to $\PP^1_k$.

\subsection{Class group of $X$}
\label{Sec_ClassGroup}
 	Because the divisor class group of $X$ is the grading group of the Cox ring, a good understanding of its structure is a preliminary to the approach of the structure of $\cox(X)$. In this section, we give a description of $\clg(X)$ by generators and relations. There are two classes of embeddings where this description is easy:

\begin{prop}\label{Prop_Desc_CLGroup_easyCase}
Suppose that $F=F_{\II}$ or is trivial. Then, $\clg(X)$ is freely generated by the classes of the $G$-stable prime divisors.
\end{prop}
\begin{proof}
Let $(D_i)_i$ denote the family of $G$-stable prime divisors in $X$. In our situation, the localization exact sequence \cite[2.2.4]{avezier_EqCoxRings} reads
\begin{center}
$0\rightarrow\bigoplus_{(D)_i}\ZZ D_i\rightarrow \clg^G(X)\rightarrow\pic^G(G/F)\rightarrow 0$.
\end{center}
Using the exact sequence \cite[2.2 (1)]{avezier_EqCoxRings}, the remark \cite[2.2.2]{avezier_EqCoxRings}, and the fact that $G$ is a semisimple simply connected algebraic group, we obtain an isomorphism $\clg^G(X)\simeq\clg(X)$. Also, there is an isomorphism $\pic^G(G/F)\simeq\hat{F}$ (\cite[4.5.1.2]{coxrings}). Finally, it suffices to remark that $\hat{F}$ is trivial (\ref{SecAppendixBinaryIcosa}).
\end{proof}

For the description of $\clg(X)$ in the general case, the ideas in \cite{BrionGroupePicardSpheriques} apply well (see also \cite{SussTorusInvariantDiv} for a similar approach in the context of $\TT$-varieties of complexity one). The fundamental remark is that in virtue of the $G$-module structure on the spaces of sections of divisorial sheaves on $X$, every Weil divisor is linearly equivalent to a $B$-stable one. From this we get an isomorphism
\begin{center}
$\clg(X)\simeq\frac{ \bigoplus_{D\in\Dd^B\sqcup\Vv(X)}\ZZ D}{\pdiv(X)^B}$,
\end{center}
where $\pdiv(X)^B$ is the group of $B$-invariant principal divisors. These divisors are the divisors of $B$-semi-invariant rational functions. After the choice of a section
\begin{center}
$\Lambda(X)=\alpha_0\ZZ\omega\rightarrow k(X)^{(B)}$, $k\alpha_0\omega\mapsto f_{\alpha_0\omega}^k$
\end{center}
of the exact sequence
\begin{center}
$1\rightarrow k(\PP^1_k)^*\rightarrow k(X)^{(B)}\rightarrow \Lambda(X) \rightarrow 1$,
\end{center}
every such semi-invariant rational function can be written uniquely as a product $gf_{\alpha_0\omega}^k$ for a certain $g\in k(\PP^1_k)^*$, and $k\in \ZZ$. The divisor of this function is then
\begin{equation}\label{Eq_GeneralForm_PrincipalDiv}
\divi(gf_{\alpha_0\omega}^k)=\sum_{z\in\PP^1_k}v_z(g)\pi^*(z)+\sum_{D\in \Dd^B\sqcup\Vv(X)}kl_DD,
\end{equation}
where $l_D:=v_D(f_{\alpha_0\omega})$ is the coordinate on $\Ee\simeq \QQ$ of $v_D$ seen in $\breve{\Ee}$, and $\pi^*$ is the pullback of Weil divisors associated to the rational quotient $\pi:X\dashrightarrow\PP^1_k$ by $B$.

\begin{prop}\label{Prop_ClassGroup_GenRelations}
The class group of $X$ is generated by the classes of exceptional divisors, and the classes of parametric divisors whose projection on $\Ee$ is non-zero. Moreover, after choosing an exceptional point $x_0\in\PP^1_k$, the relations are generated by
$$
\begin{cases}
[\pi^*(x_0)]=[\pi^*(x)],\forall x\in\PP^1_k \textrm{ exceptional},\\
\sum_{D\in\Dd^B\sqcup\Vv(X)} l_D [D]=0.
\end{cases}
$$
\end{prop}
\begin{proof}
For any parametric divisor $Z$, and exceptional point $x$, we have $[Z]=[\pi^*(x)]$ in $\clg(X)$. Also, $\pi^*(x)$ is a linear combination of exceptional divisors. It follows that $\clg(X)$ is indeed generated by the elements listed in the statement. By the general form (\ref{Eq_GeneralForm_PrincipalDiv}) of a principal divisor, every relation between these generators is a $\ZZ$-linear combination of the relations of the statement.
\end{proof}

\subsection{Singularities of $X$}
\label{Sec_SingulaX}

In this section, we give a combinatorial criterion for the singularities of $X$ to be log terminal. In fact we extend a previous result from Degtyarev for the case of a normal $G$-embedding (\cite[Thm 1]{Degtyarev}). Our method is however quite different, namely we reduce to studying singularities of normal rational affine $T$-varieties of complexity one using the particular structure of $B$-charts. This allows us to take advantage of a previous work of Liendo and Süss (\cite{LiendoSuss}), and of the description of the $T$-equivariant Cox ring for these varieties.

Consider a $G$-orbit $\mathscr{O}$ in $X$. To study the singularities along $\mathscr{O}$, it suffices to consider a $B$-chart $X_\mathscr{O}$ intersecting $\mathscr{O}$ because the translates by elements of $G$ of this chart cover $\mathscr{O}$ and are isomorphic. We examine the different types of orbits case by case, and use the classification and terminology for orbits by Luna and Vust in \cite[Section 9]{LunaVust}, that is, we consider orbits of type $A_l$ $(l\geq 1),AB,B_+,B_-,B_0$, and $C$. 

Consider the colored hyperspace $(\breve{\Ee},\Vv,\Dd^B,\varrho)$ associated with $X$, and choose coordinates on $\breve{\Ee}$ as defined in Appendix \ref{SecAppendix}. In general, a $G$-orbit is characterized by the colors and $G$-stable prime divisors containing it (\cite[16.19]{TimashevBook}). For example, an orbit $\mathscr{O}$ of type $A_l$ $(l\geq 1)$ is defined by the $G$-valuations $v_{X^{x_i}}=(x_i,h_i,l_i)\in\breve{\Ee}$, $i=1,...,l$ corresponding to the $G$-stable exceptional divisors $X^{x_i}$ containing $\mathscr{O}$ (all the colors but the exceptional colors $E^{x_i}$ contain $\mathscr{O}$, and the exceptional points $x_i$, $i=1,..,l$ are pairwise distinct). We say that an orbit of type $A_l$ $(l\geq 1)$ is \textit{Platonic} if either $l\leq 2$ or the associated tuple $(h_1,...,h_l)$ is Platonic. Furthermore, $X$ can contain at most one orbit of type $A_l$.

\begin{prop}
The singularities of $X$ are log terminal if and only if it has no $G$-fixed point and the orbit of type $A_l$, if it exists, is Platonic.
\end{prop}
\begin{proof}
For $\mathscr{O}$ a fixed point (type $B_0$), $F$ is necessarily a cyclic group $\mu_n$, and any $B$-chart $X_\mathscr{O}$ intersects all the colors, so is $G$-stable (\cite[Sec. 5]{TimashevClassification}). Hence, $X_\mathscr{O}$ is a normal affine $G/\mu_n$-embedding. By \cite[Thm 2 and Prop 4]{Panyushev}, the canonical class is torsion-free in $\clg(X_\mathscr{O})$, and $\pic(X_\mathscr{O})$ is trivial. From this, we conclude that $X_\mathscr{O}$ is not $\QQ$-Gorenstein. In particular, its singularities are not log terminal.

For orbits of type $C,AB$ and $B_+$, the chart $X_\mathscr{O}$ is given by a colored hypercone of type (A). By \ref{Cor_LocalStructThm}, there is a (trivial) $U$-torsor $X_\mathscr{O}\rightarrow Y_\mathscr{O}$, and $Y_\mathscr{O}$ is \textit{toroidal} (\cite[16.21]{TimashevBook}) in the sense of \cite[Chap IV]{KKMS}. As we consider an isolated singularity of a toroidal surface, we can suppose that $Y_\mathscr{O}$ is toric (\cite[Thm 4.2.4]{Ishii}). We conclude by using the fact that toric surface singularities are log terminal (\cite[7.4.11 and 7.4.17]{Ishii}). 

Now suppose that $\mathscr{O}$ is an orbit of type $A_l$, $l\geq 1$. Then, there are finitely many $G$-stable exceptional divisor $X^{x_1},...,X^{x_l}$ containing $\mathscr{O}$, with $x_1,...,x_l$ pairwise distinct exceptional points. We consider the $B$-chart $X_\mathscr{O}$ of $\mathscr{O}$ given by the colored hypercone of type (B) generated by the $X^{x_i}$ and the colors associated to the points of $\PP^1_k\setminus\lbrace x_1,...,x_l\rbrace$. Again, we have a $U$-torsor $X_\mathscr{O}\rightarrow Y_\mathscr{O}$. Moreover, since $Y_\mathscr{O}$ has an attractive fixed point for the $T$-action (\ref{RemBChartTypeB}), its Picard group is trivial. It follows that $Y_\mathscr{O}$ is $\QQ$-Gorenstein if and only if a multiple of $K_{Y_\mathscr{O}}$ is a principal divisor. By \cite[Prop 4.3]{LiendoSuss}, this is in turn equivalent to asking for a certain system of linear equations $Ax=y$ written in matrix form to have a solution, where $A$ has linearly independant columns (\cite[Prop 4.6]{LiendoSuss}). But in our case, $A$ is a square matrix so that $Y_\mathscr{O}$ is $\QQ$-Gorenstein. Now, we can apply the criterion \cite[Cor 5.8]{LiendoSuss} to obtain that $Y_\mathscr{O}$ (hence $X_\mathscr{O}$) has log terminal singularities if and only if the tuple $(h_1,...,h_l)$ is Platonic, where $h_i$ is the multiplicity of $X^{x_i}$ in $\pi^*(x_i)$.

Orbits of type $B_-$ can only occur when $F=\mu_n$ (\cite[5]{TimashevClassification}), we suppose that $\mathscr{O}$ is of this type. If $n\geq 3$, there are two exceptional points $x_0,x_\infty$ associated to the dense orbit and $\mathscr{O}$ lies in exactly one of the two exceptional colors, say $E^{x_0}$ (\cite[5.2]{TimashevClassification}). Consider the $B$-chart $X_\mathscr{O}$ given by the colored hypercone of type (B) generated by all the colors but $E^{x_\infty}$, and by the $G$-stable exceptional divisor sent to $x_\infty$ which contains $\mathscr{O}$. As before, we are reduced to studying the singularities of a normal rational affine $T$-surface $Y_\mathscr{O}$ of complexity one admitting an attractive fixed point, whence $\Oo(Y_\mathscr{O})^*\simeq k^*$. By \ref{ThmUinvCoxG} and \ref{Rem_ThmUinvCoxG}, the $T$-equivariant Cox ring $\cox^T(Y_\mathscr{O})$ is a polynomial ring over $k[\hat{T}]$. As a consequence, $Y_\mathscr{O}$ is toric and we conclude as above. The same method applies when $n\leq 2$.
\end{proof}

\section{Cox ring of an almost homogeneous $\SL_2$-threefold}
\label{Sec_MainCoxSL2}

Keep the notation of section \ref{Sec_SL2Threefolds_Gen}. By \cite[2.3.4 and 3.1.4]{avezier_EqCoxRings}, the $G$-equivariant Cox ring of $X$ is well-defined, finitely generated and canonically isomorphic to the ordinary Cox ring. We slightly modify the notation of \ref{ThmUinvCoxG} in order to distinguish $G$-stable prime divisors and colors. It is natural to make this distinction as colors don't depend on the embedding, whereas $G$-stable prime divisors do. We choose homogeneous coordinates on $\PP^1_k/F\simeq\PP^1_k$, and let
\begin{itemize}[label=$\bullet$]
\item $(x_i=[\alpha_i:\beta_i])_i$ be the family of exceptional points of $\pi_{|G/F}\rightarrow\PP^1_k/F$. Possible families are $\emptyset$ when $F$ is cyclic of order $n\leq 2$, $(x_0,x_\infty)$ when $F$ is cyclic of order $n\geq 3$, and $(x_v,x_f,x_e)$ for the others $F$,
\item $(x'_i=[\alpha'_i:\beta'_i])_i$ be the family whose elements are the others exceptional points of $\pi:X\dashrightarrow \PP^1_k/F$,
\item $\pi^*(x_i)=n_iE^{x_i}+\sum_j h_{ij}X^{x_i}_j$, $n_i>1$,
\item $\pi^*(x'_i)=E^{x'_i}+\sum_j h'_{ij}X^{x'_i}_j$,
\item $(s_i)_i$ (resp. $(s'_i)_i$) the family of canonical sections corresponding to the family $(E^{x_i})_i$ (resp. $(E^{x'_i})_i$),
\item  $(r_{ij})_{ij}$ (resp. $(r'_{ij})_{ij}$) the family of canonical sections corresponding to the family $(X^{x_i}_j)_{ij}$ (resp. $(X^{x'_i}_j)_{ij}$),
\item $N:=\sharp(X^{x_i}_j)_{ij}$ if $v_{X^\infty}\notin\Vv(X)$, $N:=\sharp(X^{x_i}_j)_{ij}+1$ otherwise,
\item $N':=\sharp(X^{x'_i}_j)_{ij}$,
\item  $(D^x)_{x\in\PP^1_k\setminus{\lbrace (x_i), (x'_i)\rbrace}}$ the family of parametric colors of $X$.
\end{itemize}
With this notation, the number of $G$-stable prime divisors in $X$ is $N+N'$, and we identify the subgroup of $\wdiv(X)$ generated by these divisors with $\ZZ^{N+N'}$. By \ref{ThmUinvCoxG}, we have the following presentation of $\cox(X)^U$:
\begin{itemize}[label=$\bullet$]
\item Generators: $a,b,(s_i)_i,(s'_i)_i,(r_{ij})_{ij}, (r'_{ij})_{ij}$.
\item Relations: $(\beta_i a- \alpha_i b-s_i^{n_i}\prod_j (r_{ij})^{h_{ij}})_i$, $1\leq i\leq \sharp(x_i)_i$, and $(\beta'_i a- \alpha'_i b-s'_i\prod_j (r'_{ij})^{h'_{ij}})_i$, $1\leq i\leq \sharp(x'_i)_i$.
\end{itemize}

\subsection{Characterizing log terminality in the total coordinate space}
\label{Sec_CharacterizingLogTerminality}

In this section, we provide a condition of combinatorial nature for the total coordinate space $\tilde{X}$ to have log terminal singularities. This is an interesting question, for example a $\QQ$-factorial normal projective variety $Z$ is of \textit{Fano type} (i.e. there exists an effective $\QQ$-divisor $\Delta$ such that $-(K_Z+\Delta)$ is ample and the pair $(Z,\Delta)$ is Kawamata log terminal) if and only if its Cox ring is finitely generated with log terminal singularities (\cite[Thm 1.1]{Gongyo}). In the following proposition, the condition on $\cox(X)^U$ translates into a condition on a geometric data attached to $X$ (see \ref{Rem_Inputdata}, \ref{Rem_Platonic}).

\begin{prop}
The total coordinate space $\tilde{X}$ has log terminal singularities if and only if $\cox(X)^U$ is a Platonic ring.
\end{prop}
\begin{proof}
Apply \cite[3.4.3]{avezier_EqCoxRings}.
\end{proof}

\subsection{Geometry of the special fiber}
\label{Sec_GeoSpecialFiber}

The good quotient
\begin{center}
$f:\tilde{X}\xrightarrow{//G}\AAA_k^{N+N'}$,
\end{center}
is a $\Gamma_{\clg(X)}$-equivariant morphism, where $\Gamma_{\clg(X)}$ acts on $\AAA_k^{N+N'}$ through the surjective morphism 
\begin{center}
$\Gamma_{\clg(X)}\rightarrow\GG_m^{N+N'}$
\end{center}
dually defined by the natural injective morphism $\ZZ^{N+N'}\xhookrightarrow{}\clg(X)$ which sends a $G$-stable divisor in $X$ to its class. The morphism $f$ pulls back the standard coordinates of the affine space to the canonical sections associated with the corresponding $G$-stable prime divisors. It follows from \cite[2.8.1]{avezier_EqCoxRings} that the general schematic fibers of $f$ are normal varieties isomorphic to the total coordinate space of the open orbit $G/F$. 

In this section, we study the geometry of the \textit{special fiber}, that is, the schematic zero fiber $\tilde{X}_0:=f^{-1}(0)$. By virtue of the permanency of a lot of properties when taking $U$-invariants (\cite[D.5]{TimashevBook}), we in fact study the zero fiber $\tilde{Y}_0:=\tilde{X}_0//U$ of the induced morphism
\begin{center}
$f_U:\tilde{Y}\rightarrow \AAA_k^{N+N'}$,
\end{center}
where $\tilde{Y}:=\tilde{X}//U$. We suppose that $X$ doesn't admit an exceptional divisor dominating $\PP^1_k$ ($v_{X^\infty}\notin \Vv(X)$). This is indeed harmless for our purpose as if $v_{X^\infty}\in \Vv(X)$, then one has to add the associated canonical section to the generating set of $\cox(X)^U$ from \ref{ThmUinvCoxG}, but this generator doesn't appear in any relation.

\subsubsection{$F=\mu_n$, $n\leq 2$}
\label{PropSpecialFiberTrivial}

In this case, we have $N=0$ and the presentation of $\cox(X)^U$ reads
\begin{itemize}[label=$\bullet$]
\item Generators: $a,b,(s'_i)_i, (r'_{ij})_{ij}$.
\item Relations: $(\beta'_i a- \alpha'_i b-s'_i\prod_j (r'_{ij})^{h'_{ij}})_i$, $1\leq i\leq \sharp(x'_i)_i$.
\end{itemize}
If $\sharp(x'_i)_i\leq 2$, then $\cox(X)^U$ is a polynomial $k$-algebra. Indeed, each relation can be used to remove a generator (first $a$ and then possibly $b$) from the generating set, so that we end up with a polynomial algebra. If $\sharp(x'_i)_i>2$, each new exceptional point starting from the third defines a new relation between the remaining generators. In any case, denote $\Sigma$ the new generating set. The coordinate algebra of $\tilde{Y}_0$ is $\cox(X)^U/((r'_{ij})_{ij})$, and is freely generated by the elements of the set $\Sigma\setminus \lbrace (r'_{ij})_{ij} \rbrace$.  Indeed, all the relations become trivial in this quotient. This yields the

\begin{prop}\label{PropSpecialFiberTrivial}
The special fiber $\tilde{X}_0$ is a normal variety. Moreover, $\tilde{Y}_0=\tilde{X}_0//U$ is an affine space.
\end{prop}

\begin{rem}
For a spherical variety $Z$ under a connected reductive group $G_1$, the quotient morphism
\begin{center}
$\tilde{Z}^{G_1}\rightarrow \tilde{Z}^{G_1}//G_1$
\end{center}
is faithfully flat, and the special fiber is a normal horospherical variety (\cite{Brion2007}). Both results does not extend to varieties of complexity one. Indeed, the general fibers of $f$ are isomorphic to $G$, thus of dimension three. On the other hand, when $\sharp(x'_i)_i\geq 2$, the quotient by $U$ of the special fiber is an affine space of dimension $\sharp(x'_i)_i$, thus $f$ is not flat when $\sharp(x'_i)_i\geq 3$. For a non-horospherical example, consider the case where $X=G$. Then $\tilde{X}=\tilde{X}_0=X$ which is not horospherical.
\end{rem}

\subsubsection{$F=\mu_n$, $n\geq 3$}
\label{Sec_PresGenRel_CoxU_Cyclic}
Recall (\ref{SecAppendixCyclic}), that we defined $\bar{n}:=n$ if $n$ is odd, and $\bar{n}:=n/2$ otherwise, and that the morphism $\pi_{|G/\mu_n}$ defines two exceptional points $x_0=[0:1], x_\infty=[-1:0]\in\PP^1_k$ with respect to the homogeneous coordinates $g_3^{\bar{n}},g_4^{\bar{n}}$. These points define the two relations
\begin{center}
$a=s_0^{\bar{n}}\prod_jr_{0 j}^{h_{0 j}}$, and $b=s_\infty^{\bar{n}}\prod_jr_{\infty j}^{h_{\infty j}}$.
\end{center}
Using these relations, we remove $a$ and $b$ from the set of generators. This yields the following presentation of $\cox(X)^U$:
\begin{itemize}[label=$\bullet$]
\item Generators: $s_0, s_\infty, (r_{0 j})_j,(r_{\infty j})_j, (s'_i)_i, (r'_{i j})_{i j}$.
\item Relations: $(\beta'_i s_0^{\bar{n}}\prod_jr_{0 j}^{h_{0 j}}- \alpha'_i s_\infty^{\bar{n}}\prod_jr_{\infty j}^{h_{\infty j}}=s'_i\prod_j (r'_{ij})^{h'_{ij}})_i$, $1\leq i\leq \sharp(x'_i)_i$.
\end{itemize}
Consider the coordinate algebra $\cox(X)^U/((r_{0j})_j,(r_{\infty j})_j,(r'_{ij})_{ij})$ of $\tilde{Y}_0$. The elements $s_0, s_\infty, (s'_i)_i$ generate this algebra, and we obtain the following criterion via a case by case analysis.

\begin{prop}\label{PropSpecialFiberCyclic}We have the following equivalences:
\begin{align*}
   \tilde{X}_0 \textrm{ is a normal variety} &\iff \tilde{Y}_0 \textrm{  is a normal variety }\\
    	& \iff \tilde{Y}_0 \textrm{ is an affine space}\\
    	& \iff ((X^{x_0}_j)_j\neq\emptyset \textrm{ and }(X^{x_\infty}_j)_j\neq\emptyset) \textrm{ or } (x'_i)_i=\emptyset
\end{align*}
\end{prop}

\begin{ex}
Suppose that $(X^{x_0}_j)_j=\emptyset$, $(X^{x_\infty}_j)_j=\emptyset$, and $(x'_i)_i$ consist of a unique point $x'_1=[\alpha'_1:\beta'_1]$. Then, the ideal of relations is principal, generated by the relation $\beta'_1 s_0^{\bar{n}}- \alpha'_1 s_\infty^{\bar{n}}=0$. It follows that the special fiber is a reducible reduced non-normal algebraic scheme. Indeed, $\tilde{Y}_0$ is the union of ${\bar{n}}$ planes intersecting along the line of equation $s_0=s_\infty=0$ in $k^3$ with coordinates $s_0, s_\infty, s'_1$.
\end{ex}

\begin{ex}
Suppose that $(X^{x_0}_j)_j=\emptyset$, $(X^{x_\infty}_j)_j=\emptyset$, and $(x'_i)_i$ consists of at least two points. Then, the ideal of relations is generated by $s_0^{\bar{n}}=0$ and $s_\infty^{\bar{n}}=0$. It follows that the special fiber is an irreducible non-reduced algebraic scheme.
\end{ex}

\subsubsection{$F$ is binary polyhedral}

As a typical example, we give the generators and relations when $F=F_\TT$ is the binary tetrahedral group. We can assume that the three exceptional points $x_v, x_e, x_f$ have homogeneous coordinates 
\begin{center}
$[0:1],[1:0],[-1:-1]$,
\end{center}
and obtain the following presentation of $\cox(X)^U$:
\begin{itemize}[label=$\bullet$]
\item Generators: $s_v, s_e, s_f, (r_{v,j})_j,(r_{e,j})_j, (r_{f,j})_j, (s'_i)_i, (r'_{ij})_{ij}$.
\item Relations: $s_v^3\prod_jr_{v,j}^{h_{v,j}}+s_e^2\prod_jr_{e,j}^{h_{e,j}}+s_f^3\prod_jr_{f,j}^{h_{f,j}}=0,\,\,(\beta'_i s_v^3\prod_jr_{v,j}^{h_{v,j}}+ \alpha'_i s_e^2\prod_jr_{e,j}^{h_{e,j}}=s'_i\prod_j (r'_{ij})^{h'_{ij}})_i$.
\end{itemize}
Consider the coordinate algebra 
\begin{center}
$(\cox(X))^U/((r_{v,j})_j, (r_{e,j})_j,(r_{f,j})_j,(r'_{ij})_{ij})$
\end{center}
of $\tilde{Y}_0$. The elements $s_v, s_e,s_f, (s'_i)_i$ generate this algebra, and we obtain the following criterion by again analyzing the different cases.
\begin{prop}
Suppose that $F$ is binary polyhedral. Then $\tilde{X}_0$ is a normal variety if and only if one of the following conditions is satisfied: 
\begin{itemize}
\item $(X^{x_v}_j)_j\neq\emptyset$ and $(X^{x_e}_j)_j\neq\emptyset$ and $(X^{x_f}_j)_j\neq\emptyset$
\item $(x'_i)_i=\emptyset$ and $(X^{x_v}_j)_j=(X^{x_e}_j)_j=(X^{x_f}_j)_j=\emptyset$
\end{itemize}
\end{prop}

\begin{ex}
Suppose that $(X^{x_v}_j)_j=(X^{x_e}_j)_j=(X^{x_f}_j)_j=\emptyset$, and $(x'_i)_i$ consists of a unique point $x'_1$ of homogeneous coordinates $[\alpha'_1:\beta'_1]$ (distinct from $[1:0],[0:1],[-1:-1]$). Then, the coordinate algebra of $\tilde{Y}_0$ is generated by $s_v,s_e,s_f,s'_1$ with the two relations $s_v^3+s_e^2+s_f^3=0$ and $\beta'_1 s_v^3+\alpha'_1 s_e^2=0$. We check that the special fiber is a non-normal variety. 

Let $I=(s_v^3+s_e^2+s_f^3, \beta'_1 s_v^3+\alpha'_1 s_e^2)$ be the ideal of the polynomial algebra $k[s_v,s_e,s_f,s'_1]$, we claim that it is a prime ideal. To prove this, we can work in $k[s_v,s_e,s_f]$ and replace $I$ by $I\cap k[s_v,s_e,s_f]$. Then, consider $A:=k[s_v,s_e]$, and $J:=I\cap A=(\beta'_1 s_v^3+ \alpha'_1 s_e^2)$. The algebra $B:=A/J$ is integral, and to prove the claim, it now suffices to prove that $\bar{I}=(\bar{s_v}^3+\bar{s_e}^2+s_f^3)$ is prime in $B[s_f]$. But it is clear that $\bar{s_v}^3+\bar{s_e}^2+s_f^3$ is irreducible in $\fract(B)[s_f]$, hence a prime element of this polynomial algebra. As $((\bar{s_v}^3+\bar{s_e}^2+s_f^3)\fract(B)[s_f])\cap B[s_f]=(\bar{s_v}^3+\bar{s_e}^2+s_f^3)$ we obtain that $(\bar{s_v}^3+\bar{s_e}^2+s_f^3)$ is a prime ideal. It follows that the special fiber is an affine variety. As a surface in $k^4$ with coordinates $s_v,s_e,s_f,s'_1$ having a one-dimensional singular locus, $\tilde{Y}_0$ is not normal. Indeed, the line of equation $s_v=s_e=s_f=0$ is the singular locus. Hence, the special fiber is not normal either.
\end{ex}

\subsection{$\cox(X)^U$ is the Cox ring of a $T$-variety of complexity one}
\label{SecUTorsor}

By Section \ref{SecCoxU}, $\cox(X)^U$ can be interpreted as the Cox ring of a $T$-variety $Y$ of complexity one. We build such a variety $Y$ in a natural way from $X$. The idea is to find a $B$-stable open subvariety $V$ of $X$ whose complement is of codimension $\geq 2$ in $X$, and which is a $U$-torsor over a normal rationa $T$-surface $Y$ of complexity one. Notice that we can always suppose that $V$ (hence $Y$) is smooth, up to replacing $V$ by its smooth locus. For such $V$ and $Y$, both Cox rings are well-defined, finitely generated (\cite[3.1.4]{avezier_EqCoxRings}), and $\cox(X)\simeq\cox(V)$. Moreover, using \cite[2.2 (2)]{avezier_EqCoxRings}, \cite[2.2.2]{avezier_EqCoxRings}, \cite[2.2.4]{avezier_EqCoxRings}, and \cite[2.5.2]{avezier_EqCoxRings}, we obtain isomorphisms 
\begin{center}
$\clg(X)\simeq\clg^U(X)\simeq \pic^U(V)\simeq\pic(Y)$.
\end{center}
Also by \cite[Sec. 2.10]{avezier_EqCoxRings}, we have a cartesian square 
\begin{equation}\label{Eq_Diag_CartesianUTorsors}
	\begin{tikzcd}
		\hat{V} \arrow[r,"/U"] \arrow[d,  swap, "/\Gamma_{\pic(V)}"] & \hat{Y} \arrow[d, "/\Gamma_{\pic(Y)}"]\\
		V \arrow[r, "/U"] & Y,
	\end{tikzcd}
\end{equation}
where the horizontal arrows are $U$-torsors and the vertical arrows are universal torsors. This implies that we have $\cox(Y)\simeq \cox(X)^U$, as desired.

Now, we proceed to the construction of $V$. For simplicity, denote $x_1,...,x_r\in\PP^1_k$ the exceptional points of $X$. For each $G$-stable prime divisor $X^{x_i}_j$ in $X$, we consider a $B$-chart $V_{ij}$ intersecting the open $G$-orbit in $X^{x_i}_j$, the open $G$-orbit in $X$, and no other orbits. Such a $B$-chart is given by the colored hypercone of type (A) spanned by $X^{x_i}_j$, and by all the colors but $E^{x_i}$ and $D^{x_d}$, where $x_d$ is an arbitrary fixed (distinguished) non-exceptional point. If $v_{X^\infty}\in \Vv(X)$, that is, $X$ contains a $G$-stable prime divisor $X^{\infty}$ dominating $\PP^1_k$, we consider the $B$-chart $V_\infty$ defined by the colored hypercone of type (A) spanned by $v_{X^\infty}$ and all the colors but $E^{x_1},...,E^{x_r}$, and $D^{x_d}$. Otherwise, we set $V_\infty=\emptyset$. By \ref{Cor_LocalStructThm}, we have trivial $U$-torsors
\begin{center}
$\pi_{ij}:V_{ij}\simeq U\times Y_{ij}\rightarrow Y_{ij}$,
\end{center}
where the $Y_{ij}$ are normal affine $T$-surfaces of complexity one.  If $V_\infty=\emptyset$, we set $Y_\infty=\emptyset$, otherwise we also have a trivial $U$-torsor 
\begin{center}
$\pi_\infty:V_\infty\simeq U\times Y_{\infty}\rightarrow Y_{\infty}$,
\end{center}
where $Y_\infty$ is a normal affine $T$-surface of complexity one. Finally, the open $G$-orbit $V_0\simeq G/F$ is a $U$-torsor over an affine normal $T$-surface $Y_0$ of complexity one
\begin{center}
$\pi_0:V_0\rightarrow Y_0$.
\end{center}
Indeed, $F$ acts freely on $G/U\simeq\AAA^2_k\setminus\lbrace 0\rbrace$ with closed orbits. Alternatively, we can consider two covering $B$-charts of $V_0$. Proceeding in this way, $Y_0$ is obtained by gluing two affine $T$-surfaces of complexity one $Y_{0,1}$ and $Y_{0,2}$.
\begin{prop}
The varieties $(Y_{ij})_{ij}$, $Y_\infty$, $Y_{0,1}$ and $Y_{0,2}$ glue together to a normal rational $T$-variety $Y$ of complexity one. The morphisms $(\pi_{ij})_{ij}$, $\pi_\infty$, $\pi_0$ glue together to a morphism
\begin{center}
$\pi:V:=\cup_{ij} V_{ij}\cup V_\infty\cup V_0\rightarrow Y$,
\end{center}
which is a $U$-torsor over $Y$. Moreover, $V$ is $B$-stable, with a complement in $X$ of codimension $\geq 2$.
\end{prop}
\begin{proof}
Denote $(\breve{\Ee}_T,\Vv_T)$ the hyperspace of $k(X)^U$ in which live the hypercones defining the $T$-charts $(Y_{ij})_{ij}$, $Y_\infty$. These hypercones are obtained from the colored hypercones associated with the $B$-charts $(V_{ij})_{ij}$, $V_\infty$. Indeed, they are respectively spanned by the $T$-stable divisors in these $T$-charts, and these $T$-stable divisors are the images by $\pi_{ij}$ (resp. $\pi_\infty$) of the respective intersections of the colors and $G$-stable prime divisors of $X$ with $V_{ij}$ (resp. $V_\infty$). We can define the varieties $Y_{0,1}$ and $Y_{0,2}$ in the following way: consider the subset $\Dd^B_T\subset\Vv_T$ of all the $T$-valuations obtained via $\pi_0$ from the colors of $G/F$. Then choose two distinct elements $v_{0,1}$ and $v_{0,2}$ in this set and consider the two varieties defined by the hypercones of type (A) generated respectively by $\Dd^B_T\setminus\lbrace v_{0,1}\rbrace$ and $\Dd^B_T\setminus\lbrace v_{0,2}\rbrace$. The proper non-trivial supported (hyper)faces of the hypercones defining the varieties  $(Y_{ij})_{ij}$, $Y_\infty$, $Y_{0,1}$, $Y_{0,2}$ are cones which by construction don't overlap in $\Vv_T$. By Section \ref{SecCombinatorialMaterial}, the latter varieties glue together into a normal $T$-variety of complexity one. The morphisms $(\pi_{ij})_{ij}$, $\pi_\infty$, $\pi_0$ coincide on intersections so that they glue together. The assertion that  $\pi:X\rightarrow Y$ is a $U$-torsor has already been checked locally above, and it implies that $Y$ is rational. For the last assertion, it suffices to notice that $V$ contains the open $G$-orbit and meets every boundary divisor, whence the claim on the codimension of $X\setminus V$ in $X$.
\end{proof}

\subsection{Iteration of Cox rings}
\label{Sec_Iteration}
In this section, we use the construction of the preceding section to uncover a connection between iterations of Cox rings for $X$ and $Y$. We first recall the definition of an almost principal bundle under an algebraic group. Hashimoto introduced this notion in \cite[Def. 0.4]{Hashimoto} where he systematically studies properties preserved by almost principal bundles. 

\begin{defn}\label{DefAlmostTorsor}
Let $H$ be an algebraic group, and let $Z_1,Z_2$ be normal $H$-varieties such that $H$ acts trivially on $Z_2$. We say that a $H$-equivariant morphism $\varphi:Z_1\rightarrow Z_2$ is an \textit{almost principal $H$-bundle} over $Z_2$ if there exists $H$-stable open subvarieties $V_1\subset Z_1$, $V_2\subset Z_2$ whose respective complements are of codimension $\geq 2$ and such that $\varphi$ induces a $H$-torsor $V_1\rightarrow V_2$.
\end{defn}

\begin{ex}
In the framework of Cox rings, an almost principal bundle $\varphi:Z_1\rightarrow Z_2$ under a diagonalizable group such that
\begin{itemize}
\item $Z_2$ is a normal variety with finitely generated class group and only constant invertible regular functions,
\item $\varphi$ is a good quotient,
\item $Z_1$ is a normal variety with only constant invertible homogeneous regular functions,
\end{itemize}
is precisely a \textit{quotient presentation} in the sense of \cite[4.2.1.1]{coxrings}. For example, the structural morphism of the characteristic space $\hat{Z_2}\rightarrow Z_2$, if it exists, is a quotient presentation of $Z_2$.
\end{ex}

In \cite{HausenIteration}, the authors introduce the notion of \textit{iteration of Cox rings}:  Let $Z$ be a normal variety with finitely generated Cox ring. If the total coordinate space $\tilde{Z}$ has non-trivial class group and satisfies $\Oo(\tilde{Z})^*\simeq k^*$, then it has a non-trivial well-defined Cox ring. If the latter is finitely generated, we get a new total coordinate space $\tilde{Z}^{(2)}$, and so on. This iteration process yields a sequence of Cox rings which stops if and only if one of the following cases occurs at some step:
\begin{itemize}
\item we obtain a total coordinate space whose Cox ring is not well defined (i.e. there exists $n\geq 0$ such that $\clg(\tilde{Z}^{(n)})$ has a non-trivial torsion subgroup, and $\Oo(\tilde{Z}^{(n)})^*\not\simeq k^*$).
\item we obtain a total coordinate space whose Cox ring is not finitely generated.
\item we obtain a factorial total coordinate space (i.e. with trivial class group).
\end{itemize}
If we never fall in one of the cases above, $Z$ is said to have \textit{infinite iteration of Cox rings}. Otherwise, $Z$ is said to have \textit{finite iteration of Cox rings}, and the last obtained Cox ring is the \textit{master Cox ring}. By virtue of \cite[3.4.1]{avezier_EqCoxRings}, $X$ admits finite iteration of Cox rings with a factorial finitely generated master Cox ring $\tilde{X}^{(m)}$, $m\geq 1$. In the following proposition, $Y=\tilde{Y}^{(0)}$ is the $T$ variety of complexity one constructed in the preceding section.

\begin{prop}\label{CorAlmostUTorsor}
For $1\leq i\leq m$, the categorical quotient of $\tilde{X}^{(i)}$ by $U$ identifies $\tilde{X}^{(i)}//U$ with the total coordinate space of $\tilde{Y}^{(i-1)}$. Moreover, the categorical quotient
\begin{center}
$\pi_i:\tilde{X}^{(i)}\xrightarrow{//U}\tilde{Y}^{(i)}$
\end{center}
is an almost principal $U$-bundle. 
\end{prop}
\begin{proof}
By virtue of the cartesian square (\ref{Eq_Diag_CartesianUTorsors}), the categorical quotient $\pi_1:\tilde{X}\xrightarrow{//U}\tilde{X}//U$ is an almost principal $U$-bundle, and $\tilde{X}//U$ identifies with $\tilde{Y}$. Consider the categorical quotient
\begin{center}
$\pi_2:\tilde{X}^{(2)}\xrightarrow{//U}\tilde{X}^{(2)}//U$,
\end{center}
where both are affine normal varieties (\cite[D.5]{TimashevBook}). We claim that $\tilde{X}^{(2)}//U$ is the total coordinate space of $\tilde{Y}$. Indeed, $\tilde{X}^{(2)}//U$ is naturally a variety over $\tilde{Y}$ with an affine structural morphism. Moreover, the morphism $\pi_2$ can be viewed as the morphism corresponding to the graded $\Oo_{\tilde{Y}}$-algebras morphism 
\begin{center}
$(\pi_{1*}\Rr_{\tilde{X}})^U\xhookrightarrow{}\pi_{1*}\Rr_{\tilde{X}}$.
\end{center}
Using the exact sequence \cite[2.2 (1)]{avezier_EqCoxRings}, and the Proposition \cite[2.5.2]{avezier_EqCoxRings}, we can write $\Rr_{\tilde{X}}=\bigoplus_{[\Ff]\in\clg(\tilde{Y})}\pi_1^*\Ff$, which yields an isomorphism of graded $\Oo_{\tilde{Y}}$-algebras
\begin{center}
$(\pi_{1*}\Rr_{\tilde{X}})^U\simeq \Rr_{\tilde{Y}}$,
\end{center}
again by \cite[Proposition 2.5.2]{avezier_EqCoxRings}. This proves the claim, and we have a cartesian square (see \cite[Sec. 2.10]{avezier_EqCoxRings})
\begin{center}
	\begin{tikzcd}
		\tilde{X}^{(2)} \arrow[r,"\pi_2"] \arrow[d,  swap, "//\Gamma_{\clg(\tilde{X})}"] & \tilde{Y}^{(2)} \arrow[d, "//\Gamma_{\clg(\tilde{Y})}"]\\
		\tilde{X} \arrow[r, "\pi_1"] & \tilde{Y},
	\end{tikzcd}
\end{center}
where horizontal arrows are almost principal $U$-bundles, and vertical arrows are structural morphisms of characteristic spaces. Iterating this construction, we obtain the result.
\end{proof}

\begin{cor}\label{Cor_Iteration_TvarType2}
For $i=1,...,m$, the total coordinate space $\tilde{Y}^{(i)}$ is an affine normal rational variety of complexity one under a torus action, and the regular invariant functions on $\tilde{Y}^{(i)}$ are constant.
\end{cor}
\begin{proof}
Everything stems from the fact that $\tilde{X}^{(i)}\xrightarrow{//U}\tilde{Y}^{(i)}$ is an almost principal bundle, and that $\tilde{X}^{(i)}$ is almost homogeneous of complexity one under the action of a connected reductive group of the form $G\times\
\TT_i$.
\end{proof}

\begin{cor}
There is a commutative diagram
\begin{center}
\begin{tikzcd}
\tilde{X}^{(m)} \arrow[d] \arrow[r] & ... \arrow[r] & \tilde{X}^{(2)} \arrow[r] \arrow[d] & \tilde{X} \arrow[d] \\
\tilde{Y}^{(m)} \arrow[r]           & ... \arrow[r] & \tilde{Y}^{(2)} \arrow[r]           & \tilde{Y},          
\end{tikzcd}
\end{center}
where $\tilde{X}^{(m)},\tilde{Y}^{(m)}$ are factorial, the horizontal arrows are structural morphisms of characteristic spaces, the vertical arrows are almost principal $U$-bundles, and all squares are cartesian.
\end{cor}

In \cite{Wrobel1}, Hausen and Wrobel prove that a trinomial variety obtained from Construction \ref{ConsRingFactoC1} admits finite iteration of Cox rings with a finitely generated factorial master Cox ring if and only if it is rational and the tuple $(\mathfrak{l}_0,...,\mathfrak{l}_r)$ is Platonic, where $\mathfrak{l}_i$ is the greatest common divisor of the integers appearing in the exponent vector $l_i$. It is immediate to check that $\tilde{Y}$ indeed satisfies these properties. Also by \cite[Cor. 1.4]{Wrobel1}, the length of the Cox ring iteration sequence of $Y$ (hence of $X$) is determined by the tuple $(\mathfrak{l}_0,...,\mathfrak{l}_r)$, and is bounded by $4$. Below, we give another proof for the bound $m\leq 4$ which gives more precise bounds depending on the finite subgroup $F\subset G$, and uses arguments of different nature (\ref{Prop_BoundedIteration}). This yields interesting intermediate results (\ref{Lem_Iteration_Cyclic}, \ref{Lem_TorsionSubgroup_Embeds_in_Hat}, \ref{Lem_OrderPicX_Tor_AtMostTwo}, \ref{Lem_Contradiction_QuotienPresentation}). We start by recalling a useful construction for the study of iteration (see also \cite{LBraun}, and \cite[Sec. 3.5]{avezier_EqCoxRings}).

\begin{cons}\label{ConsIteration}
Using the isomorphism $\pic^G(G/F)\simeq \hat{F}$, and that $G$ is semisimple and simply connected, the localization exact sequence \cite[2.2.4]{avezier_EqCoxRings} applied to the open orbit in $X$ reads
\begin{equation}\label{Eq_ExactSequence1}
0\rightarrow\ZZ^{N+N'}\rightarrow\clg(X)\rightarrow \hat{F}\rightarrow 0.
\end{equation}
From this sequence we obtain that $\clg(X)_{\rm tor}$ embeds in $\hat{F}$ via restriction of divisorial sheaves to $G/F$. Hence, $\clg(X)_{\rm tor}$ is canonically identified with a subgroup of $X^*(F/D(F))\simeq F/D(F)$, where $D(F)$ denotes the derived subgroup of $F$. The exact sequence
\begin{center}
$0\rightarrow\clg(X)_{\rm tor}\rightarrow\clg(X)\rightarrow M\rightarrow 0$
\end{center}
translates into the factorization
\begin{center}
$\hat{X}\xrightarrow{g_1=/\TT_1} X'\xrightarrow{f_1=/\Gamma_{\clg(X)_{\rm tor}}} X$
\end{center}
of the characteristic space $q:\hat{X}\rightarrow X$, where $\TT_1:=\Gamma_M$ is a torus, and $X'$ is a normal  $G/F_1$-embedding. Similarly to \cite[2.5.12]{avezier_EqCoxRings}, one shows that $F_1$ is the intersection of the kernels of characters in $\hat{F}$ corresponding to elements of $\clg(X)_{\rm tor}$. This yields a canonical isomorphism $\clg(X)_{\rm tor}\simeq X^*(F/F_1)$. Applying \cite[2.10.1]{avezier_EqCoxRings}, we obtain that $\hat{X}^{(2)}$ is a characteristic space of both $\hat{X}$ and $X'$. Iterating this construction, one obtains a commutative diagram
\begin{center}
\begin{tikzcd}
\hat{X}^{(m)} \arrow[r, "q_m"] \arrow[d, "g_m"] \arrow[rd, "q'_m"] & ... \arrow[r, "q_3"] & \hat{X}^{(2)} \arrow[d, "g_2"] \arrow[r, "q_2"] \arrow[rd, "q'_2"] & \hat{X} \arrow[rd, "q"] \arrow[d, "g_1"] &                     \\
X^{'(m)} \arrow[r, "f_m"]                                          & ... \arrow[r, "f_3"] & X^{'(2)} \arrow[r, "f_2"]                                          & X' \arrow[r, "f_1"]                      & X                   \\
G/F_m \arrow[r] \arrow[u, hook]                                    & ... \arrow[r]        & G/F_2 \arrow[r] \arrow[u, hook]                                    & G/F_1 \arrow[r] \arrow[u, hook]          & G/F \arrow[u, hook]
\end{tikzcd}
\end{center}
where the $q_i, q'_i$ and $g_m$ are structural morphisms of characteristic spaces, the $f_i$ are quotient presentations by the finite diagonalizable groups $F_i/F_{i-1}$, and $\hat{X}^{(m)}$ is a factorial characteristic space. Notice that $X^{'(m)}=X^{'(m-1)}$ if $\pic(X^{'(m-1)})$ is torsion-free, and $\hat{X}^{(m)}=X^{'(m)}$ if $\pic(X^{'(m-1)})$ is finite. This latter case occurs if and only if $X=G/F$. The sequence of subgroups $F,F_1...,F_{m}$ yields a normal series of $F$ with abelian quotients, and each $X^{'(i)}$ is a normal $G/F_i$-embedding. 
\end{cons}

\begin{prop}\label{Prop_Ramification}
Suppose that $X$ is a normal $G/\mu_n$-embedding, and let $d$ be the order of the torsion subgroup of $\clg(X)$. Denote $\bar{n}:=n/2$ if $n$ is even, and $\bar{n}:=n$ otherwise. Similarly let $\overline{n/d}:=n/2d$ if $n/d$ is even, and $\overline{n/d}:=n/d$ otherwise. Finally, let $\tilde{d}:=\frac{\bar{n}}{\overline{n/d}}$. Consider a point $x\in\PP^1_k/\mu_n$, and a $B$-stable prime divisor $E^x$ in $X$ such that $\pi(E^x)=x$. Then,
$$
q^*(E^x)=
\begin{cases}
\hat{E}^x_1$, if $x=x_0$ or $x=x_\infty,\\
\hat{E}^x_1+...+\hat{E}^x_{\tilde{d}}$, otherwise,$
\end{cases}
$$
where the $\hat{E}^x_i$ are pairwise distinct $B$-stable prime divisors in $\hat{X}$. 
\end{prop}
\begin{proof}
Consider the commutative square
\begin{center}
\begin{tikzcd}
X' \arrow[r, "\pi_1", dashed] \arrow[d, phantom] \arrow[d, "f_1"] & \PP^1_k/\mu_{n/d} \arrow[d, "\varphi_1"] \\
X \arrow[r, "\pi", dashed]                                       & \PP^1_k/\mu_n,           
\end{tikzcd}
\end{center}
where $\pi_1,\pi$ are the rational quotients by $B$, and $\varphi_1$ is the geometric quotient of $\PP^1_k/\mu_{n/d}$ by $\mu_d$. In view of \ref{SecAppendixCyclic}, we have 
$$
\varphi_1^*(x)=
\begin{cases}
\tilde{d}x'_1$, if $x=x_0$ or $x=x_\infty,\\
x'_1+...+x'_{\tilde{d}},$ otherwise,$
\end{cases}
$$
where the $x'_i$ are pairwise distinct points of $\PP^1_k/\mu_{n/d}$. Moreover, denoting $\Sigma_{x}$ the set of $B$-stable prime divisors in the support of $f_1^*(E^x)$, the morphism $\pi_1$ defines an equivariant map of transitive $\mu_d$-sets
\begin{center}
$\Sigma_{x}\rightarrow \varphi^{-1}(x)$.
\end{center}
Using this and the fact that $f_1$ is étale, we obtain that
\begin{center}
$f_1^*(E^x)=\sum_{k=1}^{\mid\varphi^{-1}(x)\mid}E^{x'_{k,1}}+...+E^{x'_{k,l}}$,
\end{center}
where the $E^{x'_k,i}$ are pairwise distinct $B$-stable prime divisors in $X'$ satisfying respectively $\pi_1(E^{x'_k,i})=x'_k$, and $l$ is the cardinality of the orbit $\stab_{\mu_d}(x'_1).E^{x'_k,1}$. In fact, we have $l=1$ because for a given $k$, each $E^{x'_k,i}$ $i=1,...,l$ defines the same $B$-stable valuation (see Section \ref{SecCombinatorialMaterial} and notice that the map $\varrho:\Dd^B\rightarrow\breve{\Ee}$ is injective for almost homogeneous $G$-threefolds). Now, the statement follows from the observation that $q=f_1g_1$ and that $g_1$ is a torsor under a torus.
\end{proof}

\begin{lem}\label{Lem_Iteration_Cyclic}
Suppose that $X$ is a normal $G/\mu_n$-embedding. Then, $\clg(\hat{X})$ is free of rank $(\tilde{d}-1)N'$, where $d$ is the order of the torsion subgroup of $\clg(X)$, and $\tilde{d}$ is defined as in Proposition \ref{Prop_Ramification}.
\end{lem}
\begin{proof}
By \cite[2.8.6]{avezier_EqCoxRings}, the open $G\times\Gamma_{\clg(X)}$-orbit $\hat{X}_0$ in $\hat{X}$ is isomorphic to $G\times^{\mu_n}\Gamma_{\clg(X)}$, where $\mu_n$ is identified with a subgroup of $\Gamma_{\clg(X)}$ and acts on the latter by translation. Using \cite[2.5.2]{avezier_EqCoxRings} and that $G$ is semisimple and simply connected, we obtain isomorphisms
\begin{center}
 $\pic(\hat{X}_0)\simeq \pic^G(G\times^{\mu_n}\Gamma_{\clg(X)})\simeq \pic(\Gamma_{\clg(X)}/\mu_n) \simeq \pic(\GG^{N+N'}_m)=0$.
\end{center}
We deduce that $\clg(\hat{X})$ is generated by the classes of the prime divisors lying in $\hat{X}\setminus \hat{X}_0$. Consider the free $\ZZ$-module $K$ on these prime divisors. Using the notation introduced at the beginning of Section \ref{Sec_MainCoxSL2}, we claim that the relations in $K$ defining $\clg(\hat{X})$ are given by
\begin{itemize}
\item $\divi(r'_{ij})=q^*(X^{x_i}_j)=0$, $\forall i,j$,
\item $\divi(r_{0j})=q^*(X^{x_0}_j)=0$, $\forall j$,
\item $\divi(r_{\infty j})=q^*(X^{x_\infty}_j)=0$, $\forall j$,
\item $\divi(r_\infty)=q^*(X^{\infty})$, if $v_{X^{\infty}}\in \Vv(X)$.
\end{itemize}
Indeed, on one hand we have the exact sequence \cite[2.2.4]{avezier_EqCoxRings}
\begin{center}
$0\rightarrow \Oo(\hat{X})^{*G}\rightarrow \Oo(\hat{X}_0)^{*G}\xrightarrow{\divi} \bigoplus_{(D)_i} \ZZ D_i \xrightarrow{D\mapsto [\Oo_X(D)]}\clg^G(\hat{X})\simeq\clg(\hat{X})\rightarrow 0$,
\end{center}
where $(D_i)$ is the family of prime divisors lying in $\hat{X}\setminus \hat{X}_0$. In view of Proposition \ref{Prop_Ramification}, the cardinality of this family is $\tilde{d}N'+N$. On the other hand, using \cite[2.8.5]{avezier_EqCoxRings}, the $G$-invariant units in $\Oo(\hat{X})$ are constant, and the $G$-invariant units in $\Oo(\hat{X}_0)$ are Laurent monomials in the canonical sections associated to the $G$-stable prime divisors in $X$ (namely, the $r'_{ij}, r_{0j}, r_{\infty j}, r_\infty$). By Proposition \ref{Prop_Ramification} again, we see that the submodule of $K$ generated by the above relations is a direct factor of rank $N+N'$. It follows that $\clg(\hat{X})$ is free of rank $\tilde{d}N'+N-(N+N')=(\tilde{d}-1)N'$.
\end{proof}

\begin{rem}
Proposition \ref{Prop_ClassGroup_GenRelations} provides an effective method for the description of the class group of an almost homogeneous $G$-threefold $X$. Hence, the order $d$ of the torsion subgroup of $\clg(X)$ can be computed in practice. Also, for the computation of $\clg(\tilde{X})$ in the general case, recall that we have $\clg(\tilde{X})\simeq\clg(\tilde{Y})$ because $\tilde{X}\rightarrow\tilde{Y}$ is an almost principal $U$-bundle. Now it suffices to use Wrobel's computations of class groups of affine trinomial varieties in term of arithmetic data from the exponent vectors (\cite{Wrobel2}).
\end{rem}

\begin{lem}\label{Lem_TorsionSubgroup_Embeds_in_Hat}
With the notation of Construction \ref{ConsIteration}, there is a natural identification of $\clg(X')_{\rm tor}$ with a subgroup of $\clg(\hat{X})_{\rm tor}$.
\end{lem}
\begin{proof}
Using \cite[2.5.2]{avezier_EqCoxRings} and \cite[2.2 (1)]{avezier_EqCoxRings}, we obtain an exact sequence 
\begin{center}
$0\rightarrow \hat{\TT}_1\rightarrow \clg^{\TT_1}(\hat{X})\simeq\clg(X')\rightarrow\clg(\hat{X})\rightarrow 0$,
\end{center}
from which we deduce that the torsion subgroup of $\clg(X')$ embeds in that of $\clg(\hat{X})$.
\end{proof}

\begin{lem}\label{Lem_OrderPicX_Tor_AtMostTwo}
With the notation of Construction \ref{ConsIteration}, suppose that $X$ is a normal $G/F_{\DD_n}$-embedding. Then, $\clg(X')_{\rm tor}$ is identified with a subgroup of $\ZZ/n\ZZ$.
\end{lem}
\begin{proof}
As $X'\rightarrow X$ is a quotient presentation, we can suppose that $X,X'$ are smooth, and consider Picard groups instead of class groups. By \cite[2.8.6]{avezier_EqCoxRings}, the open $G\times\Gamma_{\pic(X)}$-orbit $\hat{X}_0$ in $\hat{X}$ can be identified with $G/\mu_n\times^{\mu_2\times\mu_2}\Gamma_{\pic(X)}$, where $\mu_2\times\mu_2$ is identified with a subgroup of $\Gamma_{\pic(X)}$ and acts on the latter by translation. By \cite[2.5.2]{avezier_EqCoxRings}, we have $\pic(\hat{X}_0)\simeq\pic^{\mu_2\times\mu_2}(G/\mu_n\times\Gamma_{\pic(X)})$, and we show that the forgetful morphism
\begin{center}
$\phi:\pic^{\mu_2\times\mu_2}(G/\mu_n\times\Gamma_{\pic(X)})\rightarrow \pic(G/\mu_n\times\Gamma_{\pic(X)})\simeq\ZZ/n\ZZ$
\end{center} 
has a trivial kernel. By \cite[Lemma 2.2]{KKV}, this kernel is identified with the group of classes of algebraic cocycles $H^1_{alg}(\mu_2\times\mu_2,\Oo(G/\mu_n\times\Gamma_{\pic(X)})^*)$. Now, it follows from \cite[4.1.3]{LinearizationGBrion} that $\Oo(G/\mu_n\times\Gamma_{\pic(X)})^*\simeq\Oo(\Gamma_{\pic(X)})^*$. This last fact allows to view $\ker\phi$ as a subgroup of
\begin{center}
$\pic^{\mu_2\times\mu_2}(\Gamma_{\pic(X)})\simeq\pic(\Gamma_{\pic(X)}/\mu_2\times\mu_2)=0$.
\end{center}
As a consequence, $\ker\phi$ is trivial and we obtain an injective morphism $\pic(\hat{X}_0)\rightarrow \ZZ/n\ZZ$. 

On the other hand, the localization exact sequence \cite[2.2.4]{avezier_EqCoxRings} reads in our situation
\begin{center}
$0\rightarrow \Oo(\hat{X})^{*G}\rightarrow \Oo(\hat{X}_0)^{*G}\xrightarrow{\divi} \bigoplus_{(D)_i} \ZZ D_i \xrightarrow{D\mapsto [\Oo_X(D)]}\pic(\hat{X})\rightarrow \pic(\hat{X}_0) \rightarrow 0$,
\end{center}
where $(D_i)_i$ denotes the family of prime divisors lying in the complement of $\hat{X}_0$ in $\hat{X}$. By \cite[2.8.5]{avezier_EqCoxRings}, the $G$-invariant units in $\Oo(\hat{X})$ are constant, and the $G$-invariant units in $\Oo(\hat{X}_0)$ are Laurent monomials in the canonical sections associated to the $G$-stable prime divisors in $X$. Because $q$ is the composition of torsor by a torus followed by a torsor by a finite diagonalizable group, the divisor of such a canonical section $r$ (seen as a regular function on $\hat{X}$) is a sum of pairwise distinct elements of $(D_i)_i$. Moreover, the intersection of the supports of any two such divisors is empty. It follows that the cokernel of the morphism $\Oo(\hat{X}_0)^{*G}\xrightarrow{\divi} \bigoplus_{(D)_i} \ZZ D_i$ is a free abelian group. This implies that $\pic(\hat{X})_{\rm tor}$ embeds as a subgroup of $\pic(\hat{X}_0)$, thus of $\ZZ/n\ZZ$. By Lemma \ref{Lem_TorsionSubgroup_Embeds_in_Hat}, the same holds for $\pic(X')_{\rm tor}$.
\end{proof}

\begin{lem}\label{Lem_Contradiction_QuotienPresentation}
With the notation of Construction \ref{ConsIteration}, suppose that $\clg(X)_{\rm tor}$ is cyclic. Then $\clg(X')_{\rm tor}$ is not cyclic of even order.
\end{lem}
\begin{proof}
As before we can suppose that $X,X'$ are smooth, and consider Picard groups instead of class groups. By contradiction, suppose that $\pic(X')_{\rm tor}$ is cyclic of even order. Then, there exists a unique subgroup of order two generated by a non-trivial element $[\Ll]\in\pic(X')_{\rm tor}$. Consider the natural action of $\Gamma:=\Gamma_{\pic(X)_{\rm tor}}\simeq\mu_n$ on $X'$. Let $L$ be the line bundle over $X'$ associated with $\Ll$, and let $g_1$ denote a generator of $\Gamma(k)$ viewed as an automorphism of the variety $X'$. As the pullback by $g_1$ induces an automorphism of $\pic(X')_{\rm tor}$, we have $g_1^*\Ll\simeq \Ll$. By \cite[3.4.1 (ii)]{LinearizationGBrion}, there exists an automorphism $\tilde{g}_1$ of $L$ viewed as a variety such that $\tilde{g}_1$ commutes with the $\GG_m$-action on fibers, and restricts on $X'$ to $g_1$. By \cite[3.4.1 (ii)]{LinearizationGBrion} again, $\tilde{g}^n_1$ is the multiplication in the fibers of $L$ by a non-zero scalar (recall that $\Oo(X')^*\simeq k^*$ because $X'$ is almost homogeneous under $G$). Hence, up to compose $\tilde{g}_1$ with an automorphism of $L$ viewed as a line bundle, we can suppose that $\tilde{g}_1^n=\Id_L$. By abuse, we also let $\Gamma$ denote the subgroup of the automorphism group of the variety $L$ spanned by $\tilde{g}_1$. This is a finite group acting on $L$ compatibly with the $\GG_m$-action on the fibers and the $\Gamma$-action on the base. By definition, we have built a $\Gamma$-linearization of $\Ll$, and we now prove that this yields a contradiction. Using \cite[2.3]{KKV} and the fact that $\Oo(X')^*\simeq k^*$, we obtain an exact sequence
\begin{center}
$1 \rightarrow\hat{\Gamma}\rightarrow\pic^{\Gamma}(X')\xrightarrow{\phi}\pic(X')$,
\end{center}
where the image of the morphism $\hat{\Gamma}\rightarrow\pic^{\Gamma}(X')$ is the subgroup of the (classes of) linearizations of $\Oo_{X'}$, and $\phi$ is defined by forgetting the linearization. On the other hand, we have $\pic^{\Gamma}(X')_{\rm tor}\simeq\pic(X)_{\rm tor}\simeq\hat{\Gamma}$ (\cite[2.5.2]{avezier_EqCoxRings}). It follows that the image of $\phi$ is a free abelian subgroup of $\pic(X')$. This is a contradiction as $[\Ll]$ is of order two in $\pic(X')$ and linearizable.
\end{proof}

\begin{prop}\label{Prop_BoundedIteration}
We have the following upper bounds for the length $m$ of the iteration of Cox rings:
\begin{itemize}
\item $m=0$ (i.e. $X$ is factorial) exactly when $X=G$ or $X=G/F_{\II}$,
\item $m\leq 1$ if $F$ is binary icosahedral or cyclic of order $\leq 2$,
\item $m\leq 2$ if $F$ is cyclic of order $\geq 3$,
\item $m\leq 3$ if $F$ is binary dihedral or binary tetrahedral,
\item $m\leq 4$ if $F$ is binary octahedral.
\end{itemize}
Moreover, $\hat{X}^{(m)}$ is the characteristic space of a normal almost homogeneous $G$-threefold $X^{'(m)}$ with torsion-free class group.
\end{prop}
\begin{proof}
Except for the upper bounds on $m$, the statement is clear from Construction \ref{ConsIteration}. We have $m=0$ if and only if $X$ is factorial (i.e. $\clg(X)=0$). In view of the exact sequence (\ref{Eq_ExactSequence1}), this amount to $X=G/F$ with $\hat{F}=0$, whence $F$ is trivial or icosahedral. If $F$ is trivial or icosahedral, then $\clg(X)$ is torsion-free by \ref{Prop_Desc_CLGroup_easyCase}. This in turn implies that $\hat{X}$ is factorial by \cite[1.4.1.5]{coxrings}, whence $m\leq 1$ in this case. If $F=\mu_2$, then $m=1$ by virtue of Lemma \ref{Lem_Iteration_Cyclic}. Also, if $F=\mu_n$, $n\geq 3$, then $m\leq 2$ by Lemma \ref{Lem_Iteration_Cyclic} and \cite[1.4.1.5]{coxrings}.

Now suppose that $F=F_\OO$. Then, either $X$ has torsion-free class group and $m=1$, or we obtain a normal $G/F_\TT$-embedding $X'$. Indeed, there is no other possibility as $F_\TT$ is the derived subgroup of $F_\OO$ and $F_\OO/F_\TT\simeq\mu_2$ is a simple group. Then, either $X'$ has torsion-free class group and $m=2$, or we obtain a normal $G/F_{\DD_2}$-embedding $X^{'(2)}$. Indeed, there is no other possibility as $F_{\DD_2}$ is the derived subgroup of $F_\TT$ and $F_\TT/F_{\DD_2}\simeq\mu_3$ is a simple group. The derived subgroup of $F_{\DD_2}$ is $\mu_2$ and $F_{\DD_2}/\mu_2\simeq \mu_2\times\mu_2$. In view of Lemma \ref{Lem_Contradiction_QuotienPresentation}, we now have only two possibilities
\begin{enumerate}
\item $\clg(X^{'(2)})$ is torsion-free, whence $m=3$,
\item $\clg(X^{'(2)})\simeq \ZZ/2\ZZ\times \ZZ/2\ZZ$.
\end{enumerate}
In the second case, we obtain a normal $G/\mu_2$-embedding $X^{'(3)}$, and Lemma \ref{Lem_Iteration_Cyclic} gives that $\hat{X}^{(4)}$ is factorial, thus $m=4$ in this case.

In the case $F=F_\TT$, we have $m\leq 3$ by the last paragraph. It remains to treat the case $F=F_{\DD_n}$, $n\geq 2$. We have $D(F_{\DD_n})\simeq \mu_n$ and $F_{\DD_n}/\mu_n\simeq\mu_2\times\mu_2$ when $n$ is even, and $F_{\DD_n}/\mu_n\simeq\mu_4$ when $n$ is odd. In this last case, the proper subgroups of $F_{\DD_n}$ containing its derived subgroup are cyclic. Hence, either $\clg(X)$ is torsion-free and $m=1$, or we obtain a normal embedding $X'$ of $G$ modulo a cyclic group, whence $m\leq 3$. If $n$ is even, the subgroups of $F_{\DD_n}$ containing its derived subgroup are $F_{\DD_n}$, two copies of $F_{\DD_{n/2}}$, $\mu_{2n}$, and $\mu_n$. Hence, either $\clg(X)$ is torsion-free ($m=1$), or we obtain a normal embedding $X'$ of $G$ modulo a cyclic group ($m\leq 3$), or we obtain a normal $G/F_{\DD_{n/2}}$-embedding $X'$. In this last case, $\clg(X')_{\rm tor}$ must be trivial, thus $m=2$. Indeed, the order of $\clg(X')_{\rm tor}$ cannot be $2$ by Lemma \ref{Lem_Contradiction_QuotienPresentation}. Suppose by contradiction that this order is $4$. By Lemma \ref{Lem_Contradiction_QuotienPresentation} again, we must have $\clg(X')_{\rm tor}\simeq\ZZ/2\ZZ\times\ZZ/2\ZZ$, but this is impossible in view of Lemma \ref{Lem_OrderPicX_Tor_AtMostTwo}.
\end{proof}

\subsection{Presentation of $\cox(X)$ by generators and relations}
\label{SecPresentationCox}
In this section, we develop a general strategy for the description of $\cox(X)$ by generators and relations. In the sequel, we always suppose that $X$ doesn't admit a $G$-stable divisor dominating $\PP^1_k$, and admits at least two exceptional points. These assumptions are aimed to simplify the notation, the cases left aside can be treated with the same method. To begin with, an easy consequence of Theorem \ref{ThmUinvCoxG} is the

\begin{prop}\label{Prop_HighestWeightTheory}
The Cox ring of $X$ is generated as a $k$-algebra by the simple $G$-modules spanned by the canonical sections of the exceptional divisors.
\end{prop}

This result provides us with homogeneous generators for $\cox(X)$, namely the elements of $k$-bases of these simple $G$-modules. The more difficult task is now to describe the ideal of relations between them. We use as basic tools some elementary facts from the representation theory of $G$. For each $n\geq 0$, denote $V_n$ the space of binary forms of degree $n$. These spaces are naturally $G$-modules, $G$ acting by linear change of variables. Moreover, they are up to isomorphism the simple $G$-modules (see \cite{SpringerInvariantTheory}). Also, recall the Clebsch-Gordan decomposition
\begin{center}
$V_n\otimes_k V_m\simeq V_{n+m}\oplus V_{n+m-2}\oplus V_{n+m-4}\oplus ... \oplus V_{n-m}$, $\quad n\geq m$.
\end{center}
For a $B$-stable effective divisor $E$ in $X$, let $V_{E}\subset\cox(X)$ denote the simple $G$-module spanned by the canonical section associated with $E$.

\subsubsection{Normal $G/\mu_n$-embeddings, $n\geq 1$}

In this section, $X$ is a normal $G/\mu_n$-embedding $(n\geq 1)$. Remark that if $n\leq 2$, we can still choose two distinct exceptional points $x_0, x_\infty\in\PP^1_k$ which play the same role as in the case $n\geq 3$. Indeed, we can suppose up to a $G$-equivariant automorphism of $X$ that the pairwise distinct exceptional points on $B\backslash G/\mu_n\simeq\PP^1_k$ are
\begin{center}
$x_0=[0:1],x_\infty=[1:0],x_1=[\alpha_1:\beta_1],...,x_r=[\alpha_r:\beta_r],$
\end{center}
with respect to the homogeneous coordinates $g_3^{\bar{n}},g_4^{\bar{n}}$ (\ref{SecAppendixCyclic}). The last proposition provides a surjective morphism of graded $k$-algebras
\begin{center}
$\varphi:S:=\sym_k(V_{E^{x_0}}\oplus V_{E^{x_\infty}}\oplus (\bigoplus_{i=1}^r V_{E^{x_i}}))\otimes_k k[(r_{ij})_{i\in \lbrace 0,\infty,1,...,r\rbrace;j}]\rightarrow \cox(X)$,
\end{center}
where $k[(r_{ij})_{i\in \lbrace 0,\infty,1,...,r\rbrace;j}]\subset\cox(X)$ is a polynomial $k$-algebra, $V_{E^{x_0}}\simeq V_{E^{x_\infty}}\simeq V_1$ and $V_{E^{x_i}}\simeq V_{\bar{n}}$. We view $\cox(X)$ as a graded $k$-subalgebra of the coordinate algebra of the open $G\times\Gamma_{\clg(X)}$-orbit $\hat{X}_0$ in $\hat{X}$, and use the description of this orbit provided by \cite[2.8.6]{avezier_EqCoxRings}
\begin{center}
$\hat{X}_0\simeq G\times^{\mu_n}\Gamma_{\clg(X)}$.
\end{center}
This implies an isomorphism of graded $k$-algebras
\[
\Oo(\hat{X}_0)\simeq \bigoplus_{(k\omega,[D])\in\hat{T}\times\clg(X),\, k\textrm{ mod }n=[D]_{\mid G/\mu_n}}\Oo(G)_{k\omega}^{(T)}e^{[D]},
\]
where $e^{[D]}$ denote the character of $\Gamma_{\clg(X)}$ associated with $[D]$, and $\clg(G/\mu_n)$ is identified with $\ZZ/n\ZZ$. The canonical sections
\begin{center}
$s_0,s_\infty,s_i, r_{ij}$
\end{center}
are respectively identified with
\begin{center}
$g_3e^{[E^{x_0}]},g_4e^{[E^{x_\infty}]},(\beta_i g_3^{\bar{n}}-\alpha_i g^{\bar{n}}_4)e^{[E^{x_i}]},e^{[X^{x_i}_j]}$.
\end{center}
We consider respectively the following $k$-bases of $T$-eigenvectors
\begin{itemize}
\item $(s_0,t_0):=(g_3e^{[E^{x_0}]},g_1e^{[E^{x_0}]})$
\item $(s_\infty,t_\infty):=(g_4e^{[E^{x_\infty}]},g_2e^{[E^{x_\infty}]})$
\end{itemize}
of the simple $G$-modules $V_{E^{x_0}}$ and $V_{E^{x_\infty}}$. To start with, we have the simple case where the only exceptional points are $x_0$ and $x_\infty$:

\begin{prop}\label{Prop_GenRel_Cox_Mun_rEquals0}
If $r=0$, the ideal $\ker\varphi$ is principal, generated by the $G$-invariant relation
\begin{center}
$s_\infty t_0-s_0 t_\infty=\prod_jr_{0j}^{m_{0j}}\prod_jr_{\infty j}^{m_{\infty j}}$,
\end{center}
where $(m_{0j})_{j}$,$(m_{\infty j})_{j}$ are the families of non-negative integers defined by 
\begin{center}
$[E^{x_0}]+[E^{x_\infty}]=\sum_{j}m_{0j}[X^{x_0}_j]+\sum_{j}m_{\infty j}[X^{x_\infty}_j]$.
\end{center}
\end{prop}
\begin{proof}
In this situation, $S$ is a polynomial $k$-algebra of dimension $4+\sharp((r_{0j})_j)+\sharp((r_{\infty j})_j)$, and the dimension of $\cox(X)$ is $3+\sharp((r_{0j})_j)+\sharp((r_{\infty j})_j)$. It follows that $\ker\varphi$ is a principal ideal generated by a $G$-invariant relation. The determinant $s_\infty t_0-s_0 t_\infty=e^{[E^{x_0}]+[E^{x_\infty}]}$ is a non-zero $G$-invariant homogeneous element in $\cox(X)$. Using \cite[2.8.5]{avezier_EqCoxRings}, we obtain a $G$-invariant irreducible relation as in the above statement which necessarily generates $\ker\varphi$.
\end{proof}

We now treat the case where $r\geq 1$ (i.e. $X$ admits at least three exceptional points). For this, we first show that $X$ can be assumed to satisfy the following technical conditions:
\begin{itemize}
\item the special fiber is a normal variety,
\item the class group of $X$ is torsion-free.
\end{itemize}
Recall that the special fiber is a normal variety if and only if
\begin{center}
$n\leq 2$, or $(X^{x_0}_j)_j\neq\emptyset$ and $(X^{x_\infty}_j)_j\neq\emptyset\quad$ (Propositions \ref{PropSpecialFiberTrivial} and \ref{PropSpecialFiberCyclic}).
\end{center}
Consider the $G$-stable open subvariety $V$ consisting of the orbits of codimension $\leq 1$. We have $\cox(V)\simeq \cox(X)$ because $V$ has a complement of codimension $\geq 2$ in $X$. Thus, we can replace $X$ by $V$. Also, consider an almost homogeneous $G$-threefold $Y$ having three orbits, namely the open orbit and two $G$-stable prime divisors $Y^{0}$ and $Y^{\infty}$. We can choose $Y$ so that $v_{Y^{0}}, v_{Y^{\infty}}\notin \Vv(X)$ and $\pi(Y^{0})=x_0$, $\pi(Y^{\infty})=x_\infty$. Using the valuative criterion of separation (\cite[App. B]{TimashevBook}), it is immediate to verify that $X$ and $Y$ glue together into an almost homogeneous $G$-threefold $X_1$. By \cite[2.8.1]{avezier_EqCoxRings}, we have an isomorphism
\begin{center}
$\cox(X)\simeq\cox(X_1)/(r_0-1, r_\infty-1)$,
\end{center}
where $r_0,r_\infty$ are the canonical sections respectively associated to $Y^0, Y^\infty$. Also, $X_1$ has by construction at least three exceptional points and an associated special fiber which is a normal variety. 

Now, if $X_1$ has a class group with a non-trivial torsion subgroup. Then, we can consider a smooth complete almost homogeneous $G$-threefold $X'_1$ containing the smooth locus of $X_1$ as an open $G$-stable subvariety (\cite[2.8.4]{avezier_EqCoxRings}). By \cite[3.1.7]{avezier_EqCoxRings}, $\pic(X'_1)$ is free of finite rank, and it is directly checked that $X_1'$ has at least three exceptional points, and an associated special fiber which is a normal variety. As above, we have an isomorphism
\begin{center}
$\cox(X_1)\simeq\cox(X_1')/((r_i-1)_i)$,
\end{center}
where the $r_i$ are the canonical sections associated to the $G$-stable prime divisors lying in $X_1'\setminus (X_1)_{\rm sm}$.

Possibly replacing $X$ by $X_1$ or $X_1'$, we can suppose that $X$ has a torsion free class group, at least three exceptional points, and a special fiber which is a normal variety. In Proposition \ref{Prop_RelationCoxCyclic}, we prove that certain $G$-submodules of $S$ generate the ideal $\ker\varphi$, we start by defining these $G$-submodules.
For $k,l\in\lbrace 0,\infty,1,...,r\rbrace$, consider the natural surjective morphisms of $G$-modules
$$
\begin{cases}
\varphi_{kl}:V_{E^{x_k}}\otimes_k V_{E^{x_l}}\rightarrow V_{E^{x_k}}V_{E^{x_l}}\subset \cox(X)_{[E^{x_k}]+[E^{x_l}]},\,k< l\\
\varphi_{kk}:\sym_k^2(V_{E^{x_k}})\rightarrow V_{E^{x_k}}^2\subset \cox(X)_{2[E^{x_k}]},
\end{cases}
$$
obtained via restriction of $\varphi$. Using the Clebsch-Gordan decomposition, we obtain via a direct computation that any $B$-semi-invariant in a product $V_{E^{x_k}}V_{E^{x_l}}$ is a non-zero scalar multiple of an element of the form
\begin{center}
$g_3^{n_{0,kl}}g_4^{n_{\infty,kl}} e^{[E^{x_k}]+[E^{x_l}]} \in\cox(X)_{[E^{x_k}]+[E^{x_l}]}\subset \Oo(G)e^{[E^{x_k}]+[E^{x_l}]}$.
\end{center}
As $\cox(X)^U$ is generated as a $k$-algebra by the elements 
\begin{center}
$(s_i)_{i\in \lbrace 0,\infty,1,...,r\rbrace},(r_{ij})_{i\in \lbrace 0,\infty,1,...,r\rbrace;j}$,
\end{center}
it follows from the form of this $B$-semi-invariant that it equals a monomial
\begin{center}
$s_0^{n_{0,kl}}s_\infty^{n_{\infty,kl}}\prod_{i\in \lbrace 0,\infty,1,...,r\rbrace;j} r_{ij}^{a_{i,klj}}$,
\end{center}
where the $(a_{i,klj})_{i\in \lbrace 0,\infty,1,...,r\rbrace;j}$ is the family of non-negative integers defined by
\begin{center}
$[E^{x_k}]+[E^{x_l}]=n_{0,kl}[E^{x_0}]+n_{\infty,kl}[E^{x_\infty}]+\sum_{i\in \lbrace 0,\infty,1,...,r\rbrace;j}a_{i,klj}[X^{x_i}_j]$.
\end{center}
The table below lists the $B$-weights occurring in the products $V_{E^{x_k}}V_{E^{x_l}}$. For each $B$-weight $m\omega$ in the table, the third column provides an explicit $B$-semi-invariant $x_{kl,m\omega}$. To shorten the notation, we let 
\begin{center}
$\mathbf{r}^{\mathbf{a}_{kl,m\omega}}:=\prod_{i\in \lbrace 0,\infty,1,...,r\rbrace;j} r_{ij}^{a_{i,klj}}$.
\end{center}

\begin{center}
\begin{tabular}{|m{9em}|c|c|} 
 \hline
 $G$-module & $B$-weight: $m\omega$ & $B$-semi-invariant: $x_{kl,m\omega}$ \\ 
 \hline\hline
  \multirow{2}{6em}{$V_{E^{x_0}}V_{E^{x_\infty}}$} & $2\omega$ & $s_0s_\infty$ \\ 
	& $0$ & $\mathbf{r}^{\mathbf{a}_{0\infty,0}}$ \\
 \hline
	 $V_{E^{x_0}}^2$& $2\omega$ & $s_0^2$\\
 \hline
	 $V_{E^{x_\infty}}^2$& $2\omega$ & $s_\infty^2$\\
 \hline 
\multirow{2}{6em}{$V_{E^{x_0}}V_{E^{x_i}}$ \linebreak ${\scriptstyle (1\leq i \leq r)}$} & $(\bar{n}+1)\omega$ & $s_0s_i$ \\ 
	& $(\bar{n}-1)\omega$ & $s_\infty^{\bar{n}-1} \mathbf{r}^{\mathbf{a}_{0i,(\bar{n}-1)\omega}}$ \\
 \hline
\multirow{2}{6em}{$V_{E^{x_\infty}}V_{E^{x_i}}$ \linebreak ${\scriptstyle (1\leq i \leq r)}$} & $(\bar{n}+1)\omega$ & $s_\infty s_i$ \\ 
	& $(\bar{n}-1)\omega$ & $s_0^{\bar{n}-1} \mathbf{r}^{\mathbf{a}_{\infty i,(\bar{n}-1)\omega}}$ \\
	 \hline
\multirow{2}{6em}{$V_{E^{x_i}}^2\, {\scriptstyle (1\leq i \leq r)}$} & $2{\bar{n}}\omega$ & $s_i^2$ \\ 
	& $m_p\omega,m_p=2({\bar{n}}-2p),p=1,...,\lfloor \frac{{\bar{n}}}{2}\rfloor$ & $s_0^{{\bar{n}}-2p}s_\infty^{{\bar{n}}-2p} \mathbf{r}^{\mathbf{a}_{ii,m_p\omega}}$	\\ 
 \hline
\multirow{2}{6em}{$V_{E^{x_i}}V_{E^{x_j}}$\linebreak ${\scriptstyle (1\leq i<j\leq r,\,\alpha_i\beta_j+\alpha_j\beta_i\neq 0)}$} & $2{\bar{n}}\omega$ & $s_is_j$ \\ 
	& $m_p\omega,m_p=2({\bar{n}}-p),p=1,...,{\bar{n}}$ & $s_0^{{\bar{n}}-p}s_\infty^{{\bar{n}}-p} \mathbf{r}^{\mathbf{a}_{ij,m_p\omega}}$ \\
 \hline
\multirow{2}{6em}{$V_{E^{x_i}}V_{E^{x_j}}$\linebreak ${\scriptstyle (1\leq i<j\leq r,\,\alpha_i\beta_j+\alpha_j\beta_i= 0)}$} & $2{\bar{n}}\omega$ & $s_is_j$ \\ 
	& $m_p\omega,m_p=2({\bar{n}}-(2p+1)),p=0,...,\lfloor \frac{{\bar{n}}-1}{2}\rfloor$ & $s_0^{{\bar{n}}-(2p+1)}s_\infty^{{\bar{n}}-(2p+1)} \mathbf{r}^{\mathbf{a}_{ij,m_p\omega}}$ \\[2pt]
 \hline
\end{tabular}
\end{center}
\hfill\break
For $k,l\in\lbrace 0,\infty,1,...,r\rbrace$, $0\leq k\leq l\leq \infty$, and $m\omega$ a $B$-weight occurring in $V_{E^{x_k}}V_{E^{x_l}}$, let $y_{kl,m\omega}$ be the unique $B$-semi-invariant in $S$ sent to $x_{kl,m\omega}$ by $\varphi_{kl}$. Then, consider the following $G$-submodules of $\ker\varphi$:
\begin{center}
$M_{kl}:=(\bigoplus_{m\omega}<G.(x_{kl,m\omega}-y_{kl,m\omega})>)\oplus\ker\varphi_{kl}$,
\end{center}
where $x_{kl,m\omega}$ is naturally identified with an element of $S$. Also, in view of the description of $\cox(X)^U$ by generators and relations (Section \ref{Sec_PresGenRel_CoxU_Cyclic}), define for $i=1,...,r$ the following $G$-submodules of $\ker\varphi$:
\begin{center}
$N_i:=<G.(\beta_i s_0^{\bar{n}}\otimes\prod_jr_{0j}^{h_{0j}}- \alpha_i s_\infty^{\bar{n}}\otimes\prod_jr_{\infty j}^{h_{\infty j}}-s_i\otimes\prod_j (r_{ij})^{h_{ij}})>\simeq V_{\bar{n}}$.
\end{center}

\begin{prop}\label{Prop_RelationCoxCyclic}
The ideal $\ker\varphi$ is generated by the $G$-modules
$M_{kl}$ and $N_i$.
\end{prop}
\begin{proof}
Consider the natural $\clg(X)$-grading on $S$ induced by the canonical projection $\wdiv(X)\rightarrow \clg(X)$, and let $I$ be the homogeneous ideal generated by the $G$-modules $M_{kl}$ and $N_i$. We prove that the surjective morphism of graded $k$-algebras
\begin{center}
$\bar{\varphi}:S/I\rightarrow \cox(X)$
\end{center}
induced by $\varphi$ is an isomorphism. This follows from the claim that $(S/I)^U$ is generated as a $k$-algebra by the images of the $s_i$, $r_{ij}$, $i\in\lbrace 0,\infty,1,...,r\rbrace$. Indeed, the algebra $(S/I)^U$ can then be presented as the quotient of the polynomial $k$-algebra in the elements $s_i$, $r_{ij}$ modulo an ideal $J$ of relations. In view of the presentation of $\cox(X)^U$ and of the morphism
\begin{center}
$\bar{\varphi}_U:(S/I)^U\rightarrow \cox(X)^U$
\end{center}
induced by $\bar{\varphi}$, the ideal $J$ is contained in the ideal generated by $\oplus_{i=1}^r N_i^U$. As we have the reverse inclusion by definition of $I$, it follows that the morphism $\bar{\varphi}_U$ is an isomorphism, thus $\bar{\varphi}$ is so.

We now prove the above claim. Remark that because the special fiber is a normal variety, the $G$-modules $N_i$ become zero modulo $((r_{ij})_{i\in \lbrace 0,\infty,1,...,r\rbrace;j})$. The $k$-algebra $S/(I,(r_{ij})_{ij})$ is generated by the simple $G$-modules $V_{E^{x_i}}$, $i\in \lbrace 0,\infty,1,...,r\rbrace$, and it follows from Lemma \ref{Lem_BSemiInv_NotMax}, the preceding remark, and the definition of the $G$-modules $M_{kl}$, that the product of any two of them in $S/(I,(r_{ij})_{ij})$ is their Cartan product. As a consequence of a result of Kostant (see e.g. \cite[4.1 Lemme]{BrionRepGroupesExc}), the product of any two simple $G$-modules in $S/(I,(r_{ij})_{ij})$ is again their Cartan product or zero. This implies that  $(S/(I,(r_{ij})_{ij}))^U$ is generated as a $k$-algebra by the images of the $s_i$, $i\in\lbrace 0,\infty,1,...,r\rbrace$. Let $A$ be the $k$-subalgebra of $(S/I)^U$ generated by the images of the $s_i,r_{ij}$, $i\in\lbrace 0,\infty,1,...,r\rbrace$. Viewing $A$ and $(S/I)^U$ as $\clg(X)$-graded $k[(r_{ij})_{ij}]$-modules, the above description of $(S/(I,(r_{ij})_{ij}))^U$ yields
\begin{center}
$(S/I)^U=A+((r_{ij})_{ij})(S/I)^U$.
\end{center}
By Lemma \ref{Lem_FonctionsConstants_3exceptionalpoints}, the regular functions on $X$ are constant. This implies that $S/I$ (hence $(S/I)^U$ and $A$) can be endowed with a coarser positive grading such that the elements $r_{ij}$ have non-zero degree. Indeed, $S/I$ is naturally graded by the cone of effective divisors in $\clg(X)_\QQ$, but this cone contains no line as $\Oo(X)\simeq k$. Thus, we can choose convenient positive integers $a_{ij}$ such that the linear map $\clg(X)\rightarrow\ZZ$ defined by $[X^{x_i}_j]\mapsto a_{ij}$ is positive on generators. Now, the graded Nakayama lemma yields $(S/I)^U=A$, whence the claim. 
\end{proof}

\begin{cor}\label{Prop_IdealRelation_GEmbedding}
Suppose that $n\leq 2$. Then, the ideal $\ker\varphi$ is generated by the $G$-invariant relations
\begin{center}
$s_k t_l-s_l t_k=(\alpha_k\beta_l-\alpha_l\beta_k)\prod_{i\in \lbrace 0,\infty,1,...,r\rbrace;j}r^{m_{klij}}_{ij}$, $k,l\in\lbrace 0,\infty,1,...,r\rbrace$, $0\leq k< l\leq \infty$,
\end{center}
where $(m_{klij})_{ij}$ are the families of non-negative integers defined by 
\begin{center}
$[E^{x_k}]+[E^{x_l}]=\sum_{ij}m_{klij}[X^{x_i}_j]$,
\end{center}
and the simple $G$-modules $N_i$.
\end{cor}
\begin{proof}
It suffices to notice that in this case, the $M_{kk}$ are trivial, and the $M_{kl}$, $k<l$ are the $G$-invariant lines spanned by the elements $s_k t_l-s_l t_k-(\alpha_k\beta_l-\alpha_l\beta_k)\mathbf{r}^{\mathbf{a}_{kl,0}}$.
\end{proof}

\begin{lem}\label{Lem_BSemiInv_NotMax}
Let $Z$ be a normal $G/\mu_n$-embedding ($n\geq 1$) whose associated special fiber is a normal variety, and admitting three or more exceptional points. Then, the special fiber is a horospherical variety.
\end{lem}
\begin{proof}
In view of the table above, it suffices to show that the total degrees of the monomials $\mathbf{r}^{\mathbf{a}_{kl,m\omega}}$ are non-zero. This is equivalent to prove that the divisors
\begin{itemize}
\item $E^{x_0}+E^{x_\infty}$,
\item $E^{x_0}+E^{x_i}-(\bar{n}-1)E^{x_\infty}$, $i=1,...,r$,
\item $E^{x_\infty}+E^{x_i}-(\bar{n}-1)E^{x_0}$, $i=1,...,r$,
\item $2E^{x_i}-(\bar{n}-2p)(E^{x_0}+E^{x_\infty})$, $1\leq i\leq r$, $p=1,...,\lfloor \frac{\bar{n}}{2}\rfloor$,
\item $E^{x_i}+E^{x_j}-(\bar{n}-p)(E^{x_0}+E^{x_\infty})$, $1\leq i<j\leq r$, $p=1,...,\bar{n}$.
\end{itemize}
are not principal. We verify this for the divisor $E^{x_1}+E^{x_2}-k(E^{x_0}+E^{x_\infty})$, $0\leq k\leq \bar{n}-1$, where $\bar{n}=n$ is assumed to be odd. We skip the proof for the other cases which are treated similarly. We look for a $B$-semi-invariant rational function $gf^\gamma_\omega$, $\gamma\in\ZZ$, where $f_\omega$ is defined in \ref{SecAppendixCyclic}, and $g\in k(g_3^n/g_4^n)^*$, whose divisor is the above one. Necessarily, zeroes and poles of $g$ are located on the points $x_0,x_\infty,x_1,x_2,x_d$, and their orders are respectively $\alpha,\alpha,1,1,\beta$. In order to simplify the notation, we suppose that for each of the points $x_0,x_\infty,x_1,x_2$, there is exactly one $G$-stable prime divisor sent to this point, the general case being treated the same way. Coordinates on the hyperspace are defined as in Section \ref{SecAppendixCyclic}, and we set $v_{X^{x_i}}=(x_i,h_i,l_i)$, $i\in\lbrace 0,\infty,1,2\rbrace$. This gives
\begin{multline*}
\divi(gf^\gamma_\omega)=\alpha(n(E^{x_0}+E^{x_\infty})+h_0X^{x_0}+h_\infty X^{x_\infty})+E^{x_1}+h_1X^{x_1}+E^{x_2}+h_2X^{x_2}+\beta D^{x_d}\\+\gamma(D^{x_d}-\frac{n-1}{2}(E^{x_0}+E^{x_\infty})+l_0X^{x_0}+l_\infty X^{x_\infty}+l_1X^{x_1}+l_2X^{x_2}).
\end{multline*}
It follows that we must have $\beta=-\gamma$, and $\alpha n-\gamma\frac{n-1}{2}=-k$, the general solution of this last equation being of the form
\begin{center}
$(\alpha,\gamma)=(-k,-2k)+u(\frac{n-1}{2},n)$, $u\in\ZZ$.
\end{center}
Recall that we must have $2\alpha+2+\beta=0$, so that $u=2$. Now, this principal divisor reads
\begin{center}
$\divi(gf^\gamma_\omega)=E^{x_1}+E^{x_2}-k(E^{x_0}+E^{x_\infty})+(\alpha h_0+\gamma l_0)X^{x_0}+(\alpha h_\infty+\gamma l_\infty)X^{x_\infty}+(h_1+\gamma l_1)X^{x_1}+(h_2+\gamma l_2)X^{x_2}$.
\end{center}
It follows that we must impose $\frac{l_1}{h_1}=-\frac{1}{\gamma}=-\frac{1}{2(n-k)}$. This condition cannot be satisfied as soon as $k<n-1$, because we must have $\frac{l_1}{h_1}\leq -\frac{1}{2}$. Thus, we suppose that $k=n-1$, and we look at the condition $\frac{l_0}{h_0}=-\frac{\alpha}{\gamma}=0$ which cannot be satisfied.
\end{proof}

\begin{lem}\label{Lem_FonctionsConstants_3exceptionalpoints}
Let $Z$ be a normal $G/\mu_n$-embedding ($n\geq 1$) whose associated special fiber is a normal variety and with three or more exceptional points. Then, $\Oo(Z)\simeq k$.
\end{lem}
\begin{proof}
It suffices to prove that every $B$-semi-invariant regular function $u\in\Oo(X)$ is necessarily constant. Let $\alpha:=1$ if $n$ is odd, and $\alpha:=2$ otherwise. We can write $u=gf_{\alpha\omega}^m$, where $g\in k(g_3^{\bar{n}}/g_4^{\bar{n}})^*$, and $f_{\alpha\omega}$ is defined in Section \ref{SecAppendixCyclic}. Moreover the divisor of $u$ is effective by assumption, so that we have $v_D(u)\geq 0$, $\forall D\in\wdiv(X)^B$. By summing these inequalities over the set of colors we obtain the inequality
\begin{center}
$m(-\frac{1}{2}+\frac{1}{2{\bar{n}}}-\frac{1}{2}+\frac{1}{2{\bar{n}}}+1)\geq 0$,
\end{center}
which implies that $m\geq 0$. In view of Proposition \ref{PropSpecialFiberCyclic}, $Z$ admits at least three $G$-stable prime divisors $X^{x_0},X^{x_\infty},X^{x_1}$ sent to pairwise distinct exceptional points $x_0,x_\infty,x_1$. Consider the set of $B$-stable prime divisors consisting of $X^{x_0},X^{x_\infty},X^{x_1}$, and all the colors except $E^{x_0},E^{x_\infty}$, and $E^{x_1}$. By summing the inequalities over this set we obtain the inequality
\begin{center}
$m(\frac{l_0}{h_0}+\frac{l_\infty}{h_\infty}+\frac{l_1}{h_1}+1)\geq 0$,
\end{center}
where $(x_0,l_0,h_0),(x_\infty,l_\infty,h_\infty)$, and $(x_1,l_1,h_1)$ are the respective coordinates in the hyperspace associated with $X$ of the valuations $v_{X^{x_0}},v_{X^{x_\infty}}$ and $v_{X^{x_1}}$. Because $\frac{l_0}{h_0}+\frac{l_\infty}{h_\infty}+\frac{l_1}{h_1}+1<0$ (see \ref{SecAppendixCyclic}), the above inequality is equivalent to $m\leq 0$, and we obtain that $m=0$. Thus, we have that $g\in k^*$ because the divisor of $g$ has to be effective. Indeed, any non-constant rational function on $\PP^1_k/\mu_n$ admits at least one pole. We conclude that $u$ is constant.
\end{proof}

\begin{rem}
Suppose that $n\leq 2$. From the above description of $\cox(X)$, we have
\begin{center}
$\tilde{X}$ is a complete intersection $\iff$ $\sharp(x_i)_i\leq 2$ $\iff$ $\tilde{X}$ is a hypersurface.
\end{center}
Furthermore these equivalent conditions characterize when the good quotient morphism $\tilde{X}\xrightarrow{//G}\AAA^N$ is faithfully flat. Indeed, this morphism has equidimensional fibers if and only if $\sharp(x_i)_i\leq 2$ (see Section \ref{PropSpecialFiberTrivial}). Moreover, $\tilde{X}$ is Cohen-Macaulay under these assumptions, so that the quotient morphism is flat (\cite[III, Ex 10.9]{Hartshorne}).
\end{rem}

\begin{ex}
Consider a normal $G$-embedding $X$ with four exceptional points 
\begin{center}
$x_1=[1:0],x_2=[0:1],x_3=[1:1],x_4=[2:1]\in\PP^1_k$,
\end{center}
to which are sent the exceptional divisors
\begin{center}
$E^{x_1}$, $X^{x_1}$, $E^{x_2}$, $X^{x_2}$, $E^{x_3}$, $X^{x_3}$, $E^{x_4}$, and $X^{x_4}$.
\end{center}
We choose the section $\omega\mapsto f_\omega$ of the exact sequence
\begin{center}
$1\rightarrow k(\PP^1_k)^*\rightarrow k(X)^{(B)}\rightarrow\ZZ\omega\rightarrow 1$,
\end{center}
such that $\divi(f_\omega)$ is the color $E^{x_1}$ on $G$. This provides coordinates on $\breve{\Ee}$ for which we have
\begin{center}
$v_{X^{x_1}}=(x_1,2,-1), v_{X^{x_2}}=(x_2,3,-5),v_{X^{x_3}}=(x_3,1,-1)$, and $v_{X^{x_4}}=(x_4,5,-4)$.
\end{center}
By \ref{Prop_ClassGroup_GenRelations}, the generators $([E^{x_1}],[X^{x_1}],...,[E^{x_4}]$, $[X^{x_4}])$ of the class group give the following presentation matrix
\[
P:=\begin{bmatrix}
    -1 & -2 & 1 & 3 & 0 & 0 & 0 & 0\\
    -1 & -2 & 0 & 0 & 1 & 1 & 0 & 0\\
    -1 & -2 & 0 & 0 & 0 & 0 & 1 & 5\\
    1 & -1  & 0 & -5 &0 & -1& 0 & -4
    \end{bmatrix}.
\]
From this presentation of  $\clg(X)$ we obtain

\begin{multicols}{2}
\begin{itemize}[label=$\bullet$]
\item $[E^{x_1}]=[X^{x_1}]+5[X^{x_2}]+[X^{x_3}]+4[X^{x_4}]$
\item $[E^{x_2}]=3[X^{x_1}_1]+2[X^{x_2}]+[X^{x_3}]+4[X^{x_4}]$
\item $[E^{x_3}]=3[X^{x_1}_1]+5[X^{x_1}]+4[X^{x_4}]$
\item $[E^{x_4}]=3[X^{x_1}_1]+5[X^{x_1}]+[X^{x_3}]-[X^{x_4}]$
\end{itemize}
\end{multicols}
\noindent Finally, we find
\begin{center}
$\cox(X)=k[s_1,t_1,s_2,t_2,s_3,t_3,s_4,t_4, r_1,r_2,r_3,r_4]$,
\end{center}
with ideal of relations generated by the following

\begin{multicols}{3}
\begin{itemize}[label=$\bullet$]
\item $s_1t_2-s_2t_1-r^4_1r^7_2r^2_3r^{8}_4$
\item $s_1t_3-s_3t_1-r^4_1r^{10}_2r_3r^{8}_4$
\item $s_1t_4-s_4t_1-r^4_1r^{10}_2r^2_3r^{3}_4$
\item $s_2t_3-s_3t_2+r^6_1r^7_2r_3r^{8}_4$
\item $s_2t_4-s_4t_2+2r^6_1r^7_2r^2_3r^{3}_4$
\item $s_3t_4-s_4t_3+r^6_1r^{10}_2r_3r^{3}_4$
\item $s_2r_2^3+s_1r_1^2-s_3r_3$
\item $t_2r_2^3+t_1r_1^2-t_3r_3$
\item $s_2r_2^3+2s_1r_1^2-s_4r_4^5$
\item $t_2r_2^3+2t_1r_1^2-t_4r_4^5$
\end{itemize}
\end{multicols}
\end{ex}

\begin{ex}
Consider a normal $G/\mu_3$-embedding $X$ with three pairwise distinct exceptional points $x_0=[0:1],x_\infty=[1:0],x_1=[\alpha:\beta]\in\PP^1_k$, to which are sent the following exceptional divisors:
\begin{center}
$E^{x_0}$, $X^{x_0}$, $E^{x_\infty}$, $X^{x_\infty}$, $E^{x_1}$, and $X^{x_1}$.
\end{center}
We choose coordinates on the hyperspace as described in Section \ref{SecAppendixCyclic}, and the $G$-valuations associated to the $G$-invariant exceptional divisors have the following coordinates in $\breve{\Ee}$:
\begin{center}
$v_{X^{x_0}}=(x_0,1,-1), v_{X^{x_\infty}}=(x_\infty,1,-1)$, and $v_{X^{x_1}}=(x_1,1,-1)$.
\end{center}
The class group of $X$ is free of rank $3$ and the special fiber is a normal variety, so the preceding proposition applies. We consider the following $k$-bases for the simple $G$-modules generating $\cox(X)$:
\begin{itemize}
\item $V_{E^{x_0}}:(s_0,t_0):=(s_0,g_1e^{[E^{x_0}]})$,
\item $V_{E^{x_\infty}}:(s_\infty,t_\infty):=(s_\infty,g_2e^{[E^{x_\infty}]})$,
\item $V_{E^{x_1}}:(s_1,t_1,u_1,v_1):=(s_1,(\beta_1 g_3^2g_1-\alpha_1g_4^2g_2)e^{[E^{x_1}]},(\beta_1 g_3g_1^2-\alpha_1g_4g_2^2)e^{[E^{x_1}]},(\beta_1 g_1^3-\alpha_1g_2^3)e^{[E^{x_1}]})$,
\item $V_{X^{x_i}}$, $i=0,\infty,1:(r_{i}):=(e^{[X^{x_i}]})$.
\end{itemize}
We have
\begin{center}
$\cox(X)\simeq\sym_k(V_{E^{x_0}}\oplus V_{E^{x_\infty}}\oplus V_{E^{x_1}})\otimes_k k[r_0,r_\infty,r_1] \mod I$,
\end{center}
where $k[r_0,r_\infty,r_1]$ is a polynomial $k$-algebra, and $I$ is generated by the simple $G$-modules of relations given in the following table:
\begin{center}
\begin{tabular}{|c|c|c|} 
 \hline
 $G$-module & $B$-semi-invariant & $B$-weight\\ 
 \hline\hline
  $V_0\simeq M_{0\infty}\subset V_{E^{x_0}}\otimes_k V_{E^{x_\infty}}\simeq V_0\oplus V_2$ & $t_0 s_\infty-s_0 t_\infty-r_0r_\infty r_1^2$ & 0\\
 \hline
  $V_2\simeq M_{01}\subset V_{E^{x_0}}\otimes_k V_{E^{x_1}}\simeq V_4\oplus V_2$ & $t_1 s_0-s_1 t_0-\alpha s_\infty^2\otimes r_0r_\infty^2 r_1$ & $2\omega$\\
 \hline
  $V_2\simeq M_{\infty 1}\subset V_{E^{x_\infty}}\otimes_k V_{E^{x_1}}\simeq V_4\oplus V_2$& $t_1 s_\infty-s_1 t_\infty-\beta s_0^2\otimes r_0^2r_\infty r_1$ & $2\omega$\\
  \hline
  $V_2\simeq M_{11}\subset \sym_k^2(V_{E^{x_1}})\simeq V_6\oplus V_2$ & $t_1^2-s_1u_1-\alpha\beta s_0s_\infty\otimes r_0^3r_\infty^3 r_1^2$ & $2\omega$\\
 \hline
   $N_1\simeq V_3$ & $\beta s_0^3\otimes r_0- \alpha s_\infty^3\otimes r_{\infty}-s_1\otimes r_1$ & $3\omega$\\
 \hline
\end{tabular}
\end{center}
\end{ex}

\subsubsection{Example of affine almost homogeneous $G$-threefolds}
\label{SecSL2VarAffine}

Consider the case where $X$ is affine. In \cite{Popov}, Popov classifies these varieties up to isomorphism by means of numerical invariants. These results can be reinterpreted in term of the combinatorial description of these varieties. Indeed, the colored hypercone defining $X$ must have all the colors among its generators, which is only possible if $F\simeq \mu_n$ (\cite[5.2]{TimashevClassification}). We use coordinates on the hyperspace as defined in Section \ref{SecAppendixCyclic}, and define $u:=2$ if $n$ is even, and $u:=1$ otherwise. The combinatorial description of the normal affine $G/\mu_n$-embeddings shows that $X$ admits a unique $G$-orbit of codimension $1$ corresponding to a $G$-stable prime divisor $X^{x_0}$, which is sent to $x_0$ by the rational quotient $\pi:X\dashrightarrow \PP^1_k$ by $B$. Moreover, the condition
\begin{center}
$v_{X^{x_0}}=(x_0,h,l)\in\PP^1_k\times\ZZ_{>0}\times \frac{1}{u}\ZZ$, with $-\frac{1}{2}-\frac{1}{2\bar{n}}<\frac{l}{h}\leq-\frac{1}{2}$, and $h\wedge ul=1$
\end{center}
is satisfied. Remark that if $n\leq 2$, we can still choose two distinct points $x_0, x_\infty\in\PP^1_k$ which play the same role as in the case $n\geq 3$. In general, $X$ admits a $G$-fixed point, except in the case where $l/h=-1/2$. Concretely, this exceptional embedding is realized as the $G$-linearized line bundle
\begin{center}
$G\times^{T}\AAA^1_k\rightarrow G/T$,
\end{center}
associated to the $T$-action on $\AAA^1_k$ defined by the character $\omega^n$. The generators $([E^{x_d}],[E^{x_0}],[X^{x_0}],[E^{x_\infty}])$ of the class group give the presentation matrix
\[
P:=\begin{bmatrix}
    -1 & \bar{n} & h & 0\\
    -1 & 0 & 0 & \bar{n} \\
    u & -u\frac{\bar{n}-1}{2} & ul & -u\frac{\bar{n}-1}{2}
    \end{bmatrix}.
\]
Transforming this matrix to its Smith normal form by means of elementary operations on the rows and columns yields an isomorphism

\begin{equation}\label{EqIsoClassGroup}
\clg(X)\simeq \ZZ\times \ZZ/d\ZZ,
\end{equation}
where $d:=n\wedge h$ if $n$ is odd or $h+2l$ is even, and $d:=(\bar{n}+h)\wedge (\bar{n}-h)$ otherwise. Panyushev computed this class group in \cite[Thm 2]{Panyushev} taking as input Popov's numerical invariants.

By Proposition \ref{Prop_GenRel_Cox_Mun_rEquals0}, the Cox ring is generated by the elements $s_0, t_0, s_\infty, t_\infty, r_0$, and the relations are generated by a relation of the form
\begin{center}
$r_0^m=s_\infty t_0-s_0t_\infty$.
\end{center}
On the other hand, computing $L_1+L_2+\frac{2}{u}L_3$ on the rows of $P$, we obtain
\begin{center}
$[E^{x_0}]+[E^{x_\infty}]=-(h+2l)[X^{x_0}]$.
\end{center}
It follows that the sought relation is $r_0^{-(h+2l)}=s_\infty t_0-s_0t_\infty$. This presentation of the Cox ring is similar to the one obtained by Batyrev and Haddad in \cite{Haddad}.

\subsubsection{Comparison with the results of Batyrev and Haddad}

In \cite{Popov}, Popov associates to the isomophism class of $X$ a unique pair $(h_P,n)$ where $h_P\in ]0,1]$ is a rational number called the \textit{height} of $X$, and $n$ is the order of the cyclic group stabilizing a point in the open orbit. With \cite[III.4.3]{KraftBook}, the height has the following useful interpretation: the algebra of $U$-invariant regular functions on $X$ identifies with the algebra of the monoid
\begin{center}
$M_{h,n}=\lbrace (i,j)\in\ZZ^2_{>0}$ $; j\leq h_Pi$ and $n|i-j \rbrace$.
\end{center}
More precisely, by considering the injective graded morphism
\begin{center}
$\varphi:\Oo(X)^U\hookrightarrow \Oo(G/\mu_n)^U=\bigoplus_{i,j\geq0,n|i-j}\vect_k(g_3^ig_4^j)$
\end{center}
given by the restriction of functions, we identify elements of the monoid $M_{h,n}$ with the monomials in the image of $\varphi$. As in the preceding section, consider the $G$-valuation $v_{X^{x_0}}=(x_0,h,l)\in\breve{\Ee}$. By normality of $X$, a monomial $g_3^ig_4^j\in\Oo(G/\mu_n)^U$ belongs to $\Oo(X)^U$ if and only if
\begin{align*}
  v_{X^{x_0}}(g_3^ig_4^j)\geq 0 & \iff v_{X^{x_0}}(\frac{g_3^ig_4^j}{f_{u\omega}^{(i+j)/u}}f_{u\omega}^{(i+j)/u}) \geq 0\\
    	& \iff hv_{x_0}(\frac{g_3^ig_4^j}{f_{u\omega}^{(i+j)/u}}) +l(i+j)\geq 0\\
    & \iff \frac{h}{\bar{n}}(i+u\frac{\bar{n}-1}{2}\frac{i+j}{u}) +l(i+j)\geq 0\\
    & \iff j\leq\alpha i, \textrm{ where } \alpha:= \frac{h(\bar{n}+1)+2l\bar{n}}{h(1-\bar{n})-2l\bar{n}}.
\end{align*}
Remark that together with the condition $-\frac{1}{2}-\frac{1}{2\bar{n}}<\frac{l}{h}\leq-\frac{1}{2}$, we verify that $0<\alpha\leq 1$, with $\alpha=1$ exactly when $l/h=-1/2$. By unicity of the cone in $\ZZ^2$ defined  by the monomials of $\Oo(X)^U$, we have $\alpha=h_P$.

In \cite{Haddad}, Batyrev and Haddad consider in the affine space $\AAA^5_k$ endowed with coordinates $y,t_1,t_2,t_3,t_4$ the hypersurface $H_b$ defined by the equation
\begin{center}
$y^b=t_1t_4-t_2t_3$,
\end{center}
where $b:=\frac{q-p}{k}$, $k:=(q-p)\wedge n$, and $p/q=h_P$ with $p\wedge q=1$. Then, they set $a:=\frac{n}{k}$, and let $\GG_m\times\mu_a$ act on $\AAA^5_k$ by allocating the following weights to coordinates
\begin{center}
$\deg(y)=(k,0), \deg(t_1)=(-p,-1), \deg(t_2)=(-p,-1), \deg(t_3)=(q,1)$, and $\deg(t_4)=(q,1)$.
\end{center}
This action stabilizes $H_b$ and realizes it as the total coordinate space of $X$ (\cite[2.6]{Haddad}).

We now verify that this coincides with the presentation of the Cox ring given in the last section. For simplicity, we assume that $n$ is odd. By identification, we have
$$
\begin{cases}
p=\frac{1}{2(n\wedge h)}(h(n+1)+2nl)\\
q=\frac{1}{2(n\wedge h)}(h(1-n)-2nl)
\end{cases}.
$$
Then, the following identity holds
\begin{align*}
  b=\frac{q-p}{k} & = \frac{-(h+2l)n}{(n\wedge h)((q-p)\wedge n)}\\
    	& =\frac{-(h+2l)n}{(n\wedge h)(\frac{-n(h+2l)}{n\wedge h}\wedge n)}\\
    & = \frac{-(h+2l)n}{n(-(h+2l)\wedge n\wedge h)}\\
    & = \frac{-(h+2l)}{-2l\wedge n\wedge h}\\
    & =-(h+2l).
\end{align*}
Turning to the grading of the Cox ring, recall that the isomorphism
\begin{center}
$\clg(X)\simeq \ZZ\times \ZZ/(n\wedge h)\ZZ$
\end{center}
has been obtained by considering the four generators $[E^{x_d}], [E^{x_0}], [X^{x_0}], [E^{x_\infty}]$ of $\clg(X)$, and looking for a $\ZZ$-basis of the free $\ZZ$-module on these generators adapted to the submodule of relations. Let $u,v\in \ZZ$ be such that
\begin{center}
$-qu+kv=1$.
\end{center}
By the above isomorphism, the elements $[E^{x_0}], [X^{x_0}]$, and $[E^{x_\infty}]$ are sent respectively to $(-p,ub-v),(k,u)$, and $(q,v)$, where the second coordinate is computed in $\ZZ/(n\wedge h)\ZZ$. This gives us the corresponding degrees for the generators of the Cox ring.
This degrees correspond to the degrees of Batyrev and Haddad up to an automorphism $f$ of $\ZZ\times \ZZ/(n\wedge h)\ZZ$. Indeed, it suffices to set $f((1,0))=(1,k^{-1}u)$ and $f((0,1))=(0,k^{-1})$.

\section{Appendix: Colored equipment of $k(G/F)$}

\label{SecAppendix}

In this appendix, we recall from \cite{TimashevClassification} the colored equipment of $k(G/F)$ for $F$ either cyclic or one of the binary polyhedral groups in $G=\SL_2$. There is a natural bijection between the monoid of effective $B$-stable divisors in $G/F$, and the monoid $\Oo(G)^{(B\times F)}/k^*$. Indeed, it suffices to consider the well-defined morphism
\begin{center}
$\Oo(G)^{(B\times F)}/k^*\rightarrow\wdiv(X)^B, \bar{f}\mapsto\divi_{G/F}(f)$.
\end{center}
Via this bijection, $B$-stable prime divisors in $G/F$ correspond to indecomposable elements of $\Oo(G)^{(B\times F)}/k^*$. Remark that we can again see these elements as indecomposable homogeneous elements of $\cox(G/F)^{(B)}/k^*$. There exists a linear system $\Oo(G)^{(B\times F)}_{(n_0\omega,\lambda_0)}$ of dimension two over $k$ that defines the morphism $\pi_{|G/F}$, and a non-empty open subset of $\PP(\Oo(G)^{(B\times F)}_{(n_0\omega,\lambda_0)})\simeq \PP^1_k$ that gives the indecomposable elements of $\Oo(G)^{(B\times F)}/k^*$ corresponding to the parametric colors of $G/F$. The finite set of exceptional colors of $G/F$ is obtained by taking divisors of zeroes of the non-invertible $\clg(G/F)$-prime elements appearing in the decompositions of the decomposable elements of $\Oo(G)^{(B\times F)}_{(n_0\omega,\lambda_0)}\subset \cox(G/F)$.

\begin{defn}\label{DefRegSemiInv}\cite[16.1]{TimashevBook}
The $\clg(G/F)$-prime elements of $\Oo(G)^{(B\times F)}_{(n_0\omega,\lambda_0)}\subset \cox(G/F)$ are called the \textit{regular semi-invariants}. The others elements $\Oo(G)^{(B\times F)}_{(n_0\omega,\lambda_0)}$ are called the \textit{exceptional semi-invariants}. The \textit{subregular semi-invariants} are the non-invertible $\clg(G/F)$-prime $B$-semi-invariants appearing in the decomposition of the exceptional semi-invariants as a product of $\clg(G/F)$-prime elements.
\end{defn}

Any $G$-valuation or color $v$ of $k(G/F)$ is located in the hyperspace $\breve{\Ee}$ by a triple 
\begin{center}
$(x,h,l)\in \PP^1_k\times\QQ_+\times\QQ$,
\end{center}
where $x$ and $h$ are defined by the restriction of $v$ to $k(\PP^1_k)$, and the coordinate $l$ is obtained by evaluating $v$ at a fixed generator of the weight group $\Lambda(G/F)$ viewed in $k(G/F)^{(B)}$ via the choice of a section $\lambda\mapsto f_\lambda$ (\ref{SecCombinatorialMaterial}).

\subsection{$F$ is cyclic of order $n\geq 1$}
\label{SecAppendixCyclic}
We identify the character group of $F=\mu_n$ with $\ZZ/n\ZZ$. If $n\leq 2$, there is no exceptional semi-invariant. If $n\geq 3$, there are, up to non-zero scalar multiple, two subregular semi-invariants $g_3$ and $g_4$ whose respective weights are $(\omega,1\mod n),(\omega, -1\mod n)$. Let
$$\bar{n}=
\begin{cases}
n,\quad\quad n \textrm{ odd},\\
n/2,\quad n \textrm{ even}.
\end{cases}
$$
We have $\Oo(G)^{(B\times \mu_n)}_{(n_0\omega,\lambda_0)}=\Oo(G)^{(B\times \mu_n)}_{(\bar{n}\omega, \bar{n}\mod n)}$ generated by the two exceptional semi-invariant $g_3^{\bar{n}}$ and $g_4^{\bar{n}}$. There are two exceptional points $x_0, x_\infty$ and two exceptional colors $\pi^*(x_0)=\bar{n}E^{x_0}$, $\pi^*(x_\infty)=\bar{n}E^{x_\infty}$. Let $A$ be an arbitrary regular semi-invariant, and $x_d$ be the corresponding \textit{distinguished point} of $\PP^1_k/\mu_n\simeq\PP^1_k$. We define a section of $\Lambda(G/\mu_n)\rightarrow k(G/\mu_n)^{(B)}$ by the choice of the generator 
$$
\begin{cases}
f_\omega:=\frac{A}{(g_3g_4)^{(n-1)/2}}, \textrm{ of weight }\omega,\quad\quad (n \textrm{ odd)},\\
f_{2\omega}:=\frac{A^2}{(g_3g_4)^{(n/2-1)}}, \textrm{ of weight }2\omega,\,\quad (n \textrm{ even)}.
\end{cases}
$$
In order to have a uniform description of the hyperspace, the third coordinate $l_v$ of a $G$-valuation (or color) $v$ is defined by $l_v:=v(f_\omega)$ if $n$ is odd, and $l_v:=v(f_{2\omega})/2$ if $n$ is even. The elements of $\Vv_x\subset\Ee_x$ are the vectors $(x,h,l)$ in $\breve{\Ee}$ whose coordinates satisfy the inequalities $2l+h\leq 0$ for $x\neq x_d$ and $2l-h\leq 0$ for $x= x_d$. The colors are sent to the vectors $\varepsilon_x$ for $x\neq x_d,x_0,x_\infty$, to $(x_d,1,1)$ for $x_d$, and to $(x_0, \bar{n}, -(\bar{n}-1)/2)$ (resp. $(x_\infty, \bar{n}, -(\bar{n}-1)/2)$) for $x_0$ (resp. for $x_\infty$).

\subsection{$F$ is binary tetrahedral}\label{SecAppendixBinaryTetra}
The character group of $F_{\TT}$ identifies with the cyclic group of order $3$, for which we choose a generator $\zeta$. Up to a constant, there are three subregular semi-invariants $f_v,f_e,f_f$ whose respective weights are $(4\omega,\zeta),(6\omega, 1), (4\omega,\zeta^{-1})$. We have $\Oo(G)^{(B\times F_{\TT})}_{(n_0\omega,\lambda_0)}=\Oo(G)^{(B\times F_{\TT})}_{(12\omega, 1)}$ generated by the three exceptional semi-invariant $f_v^3,f_e^2,f_f^3$ with the relation $f_v^3+f_e^2+f_f^3=0$. This defines three exceptional points $x_v,x_e,x_f$, and three exceptional colors $\pi^*(x_v)=3E^{x_v}$, $\pi^*(x_e)=2E^{x_e}$, and $\pi^*(x_f)=3E^{x_f}$. We define a section of  $\Lambda(G/F_{\TT})\rightarrow k(G/F_{\TT})^{(B)}$ by the choice of the generator $f_{2\omega}:=\frac{f_vf_f}{f_e}$. The elements of  $\Vv_x\subset\Ee_x$ are the vectors $(x,h,l)$ of $\breve{\Ee}$ whose coordinates satisfy the inequalities $l+h\leq 0$ for $x\neq x_f, x_v$, and $l\leq 0$ for $x= x_f$ or $x=x_v$. The colors are sent to the vectors  $\varepsilon_x$ for $x\neq x_v,x_e,x_f$, and to $(x_v,3,1)$, $(x_e,2,-1)$, and $(x_f,3,1)$ for $\pi^{*}(x_v)$, $\pi^{*}(x_e)$ and $\pi^{*}(x_f)$.

\subsection{$F$ is binary octahedral}

The character group of $F_{\OO}$ identifies with the cyclic group of order $2$. Up to a constant, there are three subregular semi-invariants $f_v,f_e,f_f$ whose respective weights are $(8\omega,\zeta),(12\omega, 1), (6\omega,\zeta^{-1})$. We have $\Oo(G)^{(B\times F_{\OO})}_{(n_0\omega,\lambda_0)}=\Oo(G)^{(B\times F_{\OO})}_{(24\omega, 1)}$ generated by the three exceptional semi-invariant $f_v^3,f_e^2,f_f^4$ with the relation $f_v^3+f_e^2+f_f^4=0$. This defines three exceptional points $x_v,x_e,x_f$, and three exceptional colors $\pi^*(x_v)=3E^{x_v}$, $\pi^*(x_e)=2E^{x_e}$, and $\pi^*(x_f)=4E^{x_f}$. We define a section of  $\Lambda(G/F_{\OO})\rightarrow k(G/F_{\OO})^{(B)}$ by the choice of the generator $f_{2\omega}:=\frac{f_vf_f}{f_e}$. The elements of  $\Vv_x\subset\Ee_x$ are the vectors $(x,h,l)$ of $\breve{\Ee}$ whose coordinates satisfy the inequalities $l+h\leq 0$ for $x\neq x_f, x_v$, and $l\leq 0$ for $x= x_f$ or $x=x_v$. The colors are sent to the vectors  $\varepsilon_x$ for $x\neq x_v,x_e,x_f$, and to $(x_v,3,1)$, $(x_e,2,-1)$, and $(x_f,4,1)$ for $\pi^{*}(x_v)$, $\pi^{*}(x_e)$ and $\pi^{*}(x_f)$.

\subsection{$F$ is binary icosahedral}
\label{SecAppendixBinaryIcosa}
The character group of $F_{\II}$ is trivial. Up to a constant, there are three subregular semi-invariants $f_v,f_e,f_f$ whose respective weights are $(12\omega,\zeta),(30\omega, 1), (20\omega,\zeta^{-1})$. We have $\Oo(G)^{(B\times F_{\II})}_{(n_0\omega,\lambda_0)}=\Oo(G)^{(B\times F_{\II})}_{(60\omega, 1)}$ generated by the three exceptional semi-invariant $f_v^5,f_e^2,f_f^3$ with the relation $f_v^5+f_e^2+f_f^3=0$. This defines three exceptional points $x_v,x_e,x_f$, and three exceptional colors $\pi^*(x_v)=5E^{x_v}$, $\pi^*(x_e)=2E^{x_e}$, and $\pi^*(x_f)=3E^{x_f}$. We define a section of $\Lambda(G/F_{\II})\rightarrow k(G/F_{\II})^{(B)}$ by the choice of the generator $f_{2\omega}:=\frac{f_vf_f}{f_e}$. The elements of  $\Vv_x\subset\Ee_x$ are the vectors $(x,h,l)$ of $\breve{\Ee}$ whose coordinates satisfy the inequalities $l+h\leq 0$ for $x\neq x_f, x_v$, and $l\leq 0$ for $x= x_f$ or $x=x_v$. The colors are sent to the vectors $\varepsilon_x$ for $x\neq x_v,x_e,x_f$, and to $(x_v,5,1)$, $(x_e,2,-1)$, and $(x_f,3,1)$ for $\pi^{*}(x_v)$, $\pi^{*}(x_e)$ and $\pi^{*}(x_f)$.

\subsection{$F$ is binary dihedral of order $4n$, $n>1$}

The group $F_{\DD_{n}}$ is generated by the elements \[
h=\begin{bmatrix}
    \zeta & 0 \\
    0 & \zeta^{-1}
    \end{bmatrix} \textrm{ et }
   r= \begin{bmatrix}
    0 & -1 \\
    1 & 0
    \end{bmatrix},
\]
where $\zeta$ is a primitive $2n^{th}$ root of unity. A character of $F_{\DD_{n}}$ is determined by a pair of values $h\mapsto\zeta^k$, $r\mapsto i^l$, where $i^2=-1$. The subregular semi-invariants are $f_f=g_3g_4$, $f_e=g_3^n-(-ig_4)^n$, and $f_v=g_3^n+(ig_4)^n$ of respective weights $(2\omega,(1,-1)),(n\omega,(-1,-i^n))$, and $(n\omega,(-1,i^n))$. We have $\Oo(G)^{(B\times F_{\DD_{n}})}_{(n_0\omega,\lambda_0)}=\Oo(G)^{(B\times F_{\DD_{n}})}_{(2n\omega, (1,(-1)^n))}$ generated by the three exceptional semi-invariants $f_f^n,f_e^2,f_v^2$ with the relation $4(-i)^nf_f^n+f_e^2-f_v^2=0$. This defines three exceptional points $x_v,x_e,x_f$, and three exceptional colors $\pi^*(x_v)=2E^{x_v}$, $\pi^*(x_e)=2E^{x_e}$, and $\pi^*(x_f)=nE^{x_f}$. We define a section of $\Lambda(G/F_{\DD_{n}})\rightarrow k(G/F_{\DD_{n}})^{(B)}$ by the choice of the generator $f_{2\omega}:=\frac{f_vf_f}{f_e^{n-1}}$. The elements of  $\Vv_x\subset\Ee_x$ are the vectors $(x,h,l)$ of $\breve{\Ee}$ whose coordinates satisfy the inequalities $l+h\leq 0$ for $x\neq x_f, x_v$, and $l\leq 0$ for $x= x_f$ or $x=x_v$. The colors are sent to the vectors $\varepsilon_x$ for $x\neq x_v,x_e,x_f$, and to $(x_v,2,1)$, $(x_e,2,1)$, and $(x_f,n,1-n)$ for $\pi^{*}(x_v)$, $\pi^{*}(x_e)$ and $\pi^{*}(x_f)$.

\nocite{*}
\bibliographystyle{apa}
\bibliography{biblio/biblio}

\begin{thebibliography}{}

\bibitem[\protect\astroncite{Arzhantsev et~al.}{2018}]{HausenIteration}
Arzhantsev, I., Braun, L., Hausen, J., and Wrobel, M. (2018).
\newblock Log terminal singularities, {P}latonic tuples and iteration of {C}ox
  rings.
\newblock {\em European Journal of Mathematics}, 4, no. 1:242--312.

\bibitem[\protect\astroncite{Arzhantsev et~al.}{2014}]{coxrings}
Arzhantsev, I., Derenthal, U., Hausen, J., and Laface, A. (2014).
\newblock {\em Cox Rings}.
\newblock Cambridge University Press.

\bibitem[\protect\astroncite{Batyrev and Haddad}{2008}]{Haddad}
Batyrev, V. and Haddad, F. (2008).
\newblock On the geometry of {SL\subscript{2}}-equivariant flips.
\newblock {\em Moscow Math Journal}, 8, no. 4:621–646.

\bibitem[\protect\astroncite{Blomer et~al.}{2020}]{ManinPeyre}
Blomer, V., Brüdern, J., Derenthal, U., and Gagliardi, G. (2020).
\newblock {The {M}anin-{P}eyre conjecture for smooth spherical {F}ano varieties
  of semisimple rank one}.
\newblock {\em arXiv e-prints}.

\bibitem[\protect\astroncite{{Braun}}{2019}]{LBraun}
{Braun}, L. (2019).
\newblock {Gorensteinness and iteration of {C}ox rings for Fano type
  varieties}.
\newblock {\em arXiv e-prints}, page arXiv:1903.07996.

\bibitem[\protect\astroncite{Brion}{1985}]{BrionRepGroupesExc}
Brion, M. (1985).
\newblock Représentations exceptionnelles des groupes semi-simples.
\newblock {\em Annales Scientifiques de l'Ecole Normale Supérieure}, 18, no.
  4:345--387.

\bibitem[\protect\astroncite{Brion}{1989}]{BrionGroupePicardSpheriques}
Brion, M. (1989).
\newblock Groupe de {P}icard et nombres caractéristiques des variétés
  sphériques.
\newblock {\em Duke Mathematical Journal}, 58, no. 2:397--424.

\bibitem[\protect\astroncite{Brion}{2007}]{Brion2007}
Brion, M. (2007).
\newblock The total coordinate ring of a wonderful variety.
\newblock {\em Journal of Algebra}, 313:61--99.

\bibitem[\protect\astroncite{Brion}{2018}]{LinearizationGBrion}
Brion, M. (2018).
\newblock Linearization of algebraic group actions.
\newblock {\em Handbook of Group Actions}, Volume IV.

\bibitem[\protect\astroncite{Brion et~al.}{1986}]{BLV}
Brion, M., Luna, D., and Vust, T. (1986).
\newblock Espaces homogènes sphériques.
\newblock {\em Inventiones Mathematicae}, 84:617--632.

\bibitem[\protect\astroncite{Degtyarev}{1995}]{Degtyarev}
Degtyarev, M. (1995).
\newblock {SL\subscript{2}}-embeddings with log-terminal singularities.
\newblock {\em St-Petersburg Mathematical Journal}, 6, no. 1:113--119.

\bibitem[\protect\astroncite{Gongyo et~al.}{2015}]{Gongyo}
Gongyo, Y., Okawa, S., Sannai, A., and Takagi, S. (2015).
\newblock {C}haracterization of varieties of {F}ano type via singularities of
  {C}ox rings.
\newblock {\em Journal of Algebraic Geometry}, 24 no. 1:159--182.

\bibitem[\protect\astroncite{Hartshorne}{1977}]{Hartshorne}
Hartshorne, R. (1977).
\newblock {\em Algebraic Geometry}.
\newblock Springer.

\bibitem[\protect\astroncite{Hashimoto}{2015}]{Hashimoto}
Hashimoto, M. (2015).
\newblock Equivariant class group. {III}. almost principal fiber bundles.
\newblock \url{https://arxiv.org/abs/1503.02133}.

\bibitem[\protect\astroncite{Hausen and Wrobel}{2018}]{Wrobel1}
Hausen, J. and Wrobel, M. (2018).
\newblock On iteration of {C}ox rings.
\newblock {\em Journal of Pure and Applied Algebra}, 222, no. 9:2737--2745.

\bibitem[\protect\astroncite{Ishii}{2018}]{Ishii}
Ishii, S. (2018).
\newblock {\em Introduction to {S}ingularities}.
\newblock Springer.

\bibitem[\protect\astroncite{Kempf et~al.}{1973}]{KKMS}
Kempf, G., Knudsen, F., Mumford, D., and Saint-Donat, B. (1973).
\newblock {\em Toroidal embeddings}.
\newblock Springer - Lecture Notes in Mathematics.

\bibitem[\protect\astroncite{Knop}{1993}]{Knop93}
Knop, F. (1993).
\newblock Uber {B}ewertungen, welche unter einer reduktiven {G}ruppe invariant
  sind.
\newblock {\em Mathematische Annalen}, 295:333--363.

\bibitem[\protect\astroncite{Knop et~al.}{1989}]{KKV}
Knop, F., Kraft, H., and Vust, T. (1989).
\newblock The {P}icard group of a {G}-variety.
\newblock {\em Algebraische Transformationsgruppen und Invariantentheorie},
  pages 77--89.

\bibitem[\protect\astroncite{Kraft}{1983}]{KraftBook}
Kraft, H. (1983).
\newblock {\em Geometrische methoden in der Invariantentheorie}.
\newblock Friedr Vieweg Sohns.

\bibitem[\protect\astroncite{Liendo and Süss}{2013}]{LiendoSuss}
Liendo, A. and Süss, H. (2013).
\newblock Normal singularities with torus actions.
\newblock {\em Tohoku Math. Journal}, 65, no. 1:105--130.

\bibitem[\protect\astroncite{Luna and Vust}{1983}]{LunaVust}
Luna, D. and Vust, T. (1983).
\newblock Plongements d'espaces homogènes.
\newblock {\em Comment. Math. Helvetici.}, 58:186--245.

\bibitem[\protect\astroncite{Moser-Jauslin}{1987}]{LMJ}
Moser-Jauslin, L. (1987).
\newblock {\em Normal embeddings of {SL\subscript{2}/\textGamma}}.
\newblock PhD thesis, University of {G}eneva.

\bibitem[\protect\astroncite{Panyushev}{1992}]{Panyushev}
Panyushev, D. (1992).
\newblock The canonical module of a quasihomogenous normal affine
  {SL\subscript{2}}-variety.
\newblock {\em Math. USSR Sb.}, 73 no 2:569--578.

\bibitem[\protect\astroncite{Panyushev}{1995}]{Panyushev2}
Panyushev, D. (1995).
\newblock On homogeneous spaces of rank one.
\newblock {\em Indag. Mathema.}, 6 (3),315-323:569.

\bibitem[\protect\astroncite{Petersen and Süss}{2011}]{SussTorusInvariantDiv}
Petersen, L. and Süss, H. (2011).
\newblock Torus invariant divisors.
\newblock {\em Israel Journal of Mathematics}, 102:481–504.

\bibitem[\protect\astroncite{Ponomareva}{2015}]{Ponomareva}
Ponomareva, E.~V. (2015).
\newblock Invariants of the {C}ox rings of low-complexity double flag varieties
  for classical groups.
\newblock {\em Trans. Moscow Math. Soc.}, 76:71--133.

\bibitem[\protect\astroncite{Popov}{1973}]{Popov}
Popov, V.~L. (1973).
\newblock Quasihomogeneous affine algebraic varieties of the group
  {SL\subscript{2}}.
\newblock {\em Math. USSR Izv.}, 7:93.

\bibitem[\protect\astroncite{Springer}{1977}]{SpringerInvariantTheory}
Springer, T. (1977).
\newblock {\em Invariant Theory}.
\newblock Springer.

\bibitem[\protect\astroncite{Timashev}{1997}]{TimashevClassification}
Timashev, D. (1997).
\newblock Classification of {G}-varieties of complexity one.
\newblock {\em Izvestiya Mathematics}, 61:127--162.

\bibitem[\protect\astroncite{Timashev}{2000}]{TimashevCartier}
Timashev, D. (2000).
\newblock Cartier divisors and geometry of normal {G}-varieties.
\newblock {\em Transformation Groups}, 5, no 2:181--204.

\bibitem[\protect\astroncite{Timashev}{2011}]{TimashevBook}
Timashev, D. (2011).
\newblock {\em Homogeneous spaces and equivariant embeddings}.
\newblock Springer.

\bibitem[\protect\astroncite{{Vezier}}{2020}]{avezier_EqCoxRings}
{Vezier}, A. (2020).
\newblock {Equivariant {C}ox ring}.
\newblock {\em arXiv e-prints}, page arXiv:xxxx.xxxx.

\bibitem[\protect\astroncite{Wrobel}{2020}]{Wrobel2}
Wrobel, M. (2020).
\newblock Divisor class groups of rational trinomal varieties.
\newblock {\em Journal of Algebra}, 542:43--64.

\end{thebibliography}

\end{document}